\documentclass[11pt]{amsart}
\usepackage{amsmath,amssymb, amsthm,delarray}
\usepackage{graphicx}
\numberwithin{equation}{section}
\newtheorem{thm}{Theorem}[section]
\newtheorem{lemma}[thm]{Lemma}
\newtheorem{remark}[thm]{Remark}
\newtheorem{corollary}[thm]{Corollary}
\newtheorem{proposition}[thm]{Proposition}
{\rm}
\newtheorem{example}{Example}{\rm}
{\rm}

\def\beq{\begin{equation} }
\def\eeq{\end{equation} }
\def\e{\hbox{\rm e}}

\def\C{\mathbb{C}}

\def\N{\mathbb{N}}
\def\M{\mathbf{M}}
\def\R{\mathbb{R}}
\def\A{\mathbb{A}}
\def\P{\mathbf{P}}
\def\si{\mathbf{\Sigma}}
\def\K{\mathbf{K}}
\def\B{\mathbf{B}}
\def\Q{\mathbf{Q}}
\def\M{\mathbf{M}}

\def\G{\mathcal{G}}

\def\v{\mathbf{v}}
\def\f{\mathbf{f}}
\def\h{\mathbf{h}}

\def\u{\mathbf{u}}
\def\f{\mathbf{f}}
\def\w{\mathbf{w}}

\def\x{\mathbf{x}}

\def\y{\mathbf{y}}
\def\z{\mathbf{z}}
\def\s{\mathcal{S}}

\def\dis{\displaystyle}

\def\supp{{\rm supp}\,}

\def\R{\mathbb{R}}
\def\B{\mathbf{B}}
\def\C{\mathbf{C}}
\def\X{\mathbf{X}}

\def\det{{\rm det}\,}
\def\vol{{\rm vol}\,}
\def\dis{\displaystyle}
\def\tv{\tilde{\v}}

\begin{document}
\title[positively homogeneous functions]{Level sets and non Gaussian integrals of positively homogeneous functions}
\author{J.B. Lasserre}
\address{LAAS-CNRS and Institute of Mathematics\\
University of Toulouse,
7 Avenue du Colonel Roche, 310777 Toulouse Cedex 4, France.}
\email{lasserre@laas.fr}

\begin{abstract}
We investigate various properties
of the sublevel set $\{\x \,:\,g(\x)\leq 1\}$ 
and the integration of $h$ on this sublevel set when $g$ and $h$
are positively homogeneous functions. For instance, the latter integral
reduces to integrating $h\exp(-g)$ on the whole space $\R^n$ (a non Gaussian integral)
and when $g$ is a polynomial, then the volume of the sublevel set is a convex function
of the coefficients of $g$. In fact, whenever $h$ is nonnegative,
the functional $\int \phi(g(\x))h(\x)d\x$ is a convex function of $g$ for a large class of 
functions $\phi:\R_+\to\R$.
We also provide a numerical approximation scheme to
compute the volume or integrate $h$  (or, equivalently to approximate the associated non Gaussian integral).
We also show that finding the sublevel set $\{\x \,:\,g(\x)\leq 1\}$ 
of minimum volume that contains some given subset $\K$ is a (hard) convex optimization problem
for which we also propose two convergent numerical schemes. Finally, we provide a Gaussian-like property
of non Gaussian integrals for homogeneous polynomials that are sums of squares and critical points of a
specific function.
\end{abstract}


\maketitle


\section{Introduction}

{\it Positively homogeneous} functions (PHF) form a wide class of functions that one encounters in many applications.
As a consequence of homogeneity, they enjoy very particular properties, and among them the celebrated 
and very useful Euler's identity which allows to deduce additional properties of PHFs
in various contexts. One goal of this paper is to show that another not well-known property of PHFs
yields surprising and unexpected results, some of them already known in particular cases.
Namely, we are  concerned with PHFs, their sublevel sets and in particular, the integral
\begin{equation}
\label{def-theta}
y\mapsto I_{g,h}(y)\,:=\,\int _{\{\x\,:\,g(\x)\leq y\}}h(\x)\,d\x,\end{equation}
as a function $I_{g,h}:\R_+\to\R$ when $g,h$ are PHFs.
With $y$ fixed, we are also interested in  $I_{g,h}(y)$ now as a function of $g$, especially when $g$ is a nonnegative homogeneous polynomial.
Nonnegative homogeneous polynomials are particularly interesting as they can be used to approximate norms; see e.g. 
Barvinok \cite[Theorem 3.4]{barvinok}.
Interestingly, the integral (\ref{def-theta}) is related in a simple and remarkable manner to the non-Gaussian integral
$\int_{\R^n}h\exp(-g)d\x$ and therefore, any information on either one can help understanding and evaluating the other.

Functional integrals appear frequently in quantum Physics. In the important particular case  when $h=1$,
Morosov et Shakirov \cite{morosov1}
proved that for all {\it forms} $g$ of degree $k$, 
\begin{equation}
\label{first}
I_{g,1}(1)={\rm cte}(k)\cdot\int_{\R^n} \exp(-g)d\x,\end{equation}
where the constant depends only on $k$. 
In fact, a formula of exactly the same flavor was
already known for convex sets, and was the initial motivation and starting point of this paper.
Namely, 
if $C\subset\R^n$ is convex, its support function 
$\sigma_C$ is a PHF of degree $1$, and the polar 
$C^\circ\subset\R^n$ of $C$ is the convex set $\{\x\,:\,\sigma_C(\x) \leq 1\}$. Then
\[\vol(C^\circ)\,=\,\frac{1}{n{\rm !}}\dis\int_{\R^n}\exp(-\sigma_C(\x))\,d\x,\qquad \forall C.\]
(See e.g. Hiriart-Urruty \cite[Exercise 65, p. 243-244]{jbhu} and Barvinok \cite[Problem 4, p. 207]{barvinok}.)
In fact, as we will see, the intriguing result (\ref{first}) is a particular case of a more general and striking result about PHFs.

The connection between 
$I_{g1}$ and $\int \exp(-g)$ is a consequence of what 
the authors in \cite{morosov1} call  {\it action-independence} of integral discriminants. That is, for a very general class of univariate functions $\phi$, and for every form $g$ of degree $d$,
\begin{equation}
\label{state}
\int \phi(g(\x))d\x\,=\,C(\phi,d)\cdot\theta(g),\end{equation}
for some function $\theta$, where the constant $C(\phi,d)$ depends only on $\phi$ and $d$; see \cite[(73)]{morosov1}.
And so for instance, with the choices $t\mapsto \phi_1(t)=1_{[0,1]}(t)$ and
$\phi_2(t)=\exp(-t)$, and with $\R^n$ as integration contour,
\[\int_{\{\x\,:\,g(\x)\leq 1\}}d\x\,=\,\frac{C(\phi_1,d)}{C(\phi_2,d)}\,\int_{\R^n} \exp(-g)\,d\x,\qquad \forall g,\]
whenever any of the above integrals finite.
In Morosov and Shakirov \cite{morosov1} the motivation of the authors was not to connect
$I_{g1}$ with $\int \exp(-g)$ but rather to study
a specific class of non-Gaussian integrals,
the so-called {\it integral discriminants}. 
In fact, the theory of integral discriminants 
which lies in the field of non linear algebra, extends to integrals
on more general contours of integration, and
Ward identities permits to study properties of such integrals that are invariant under a change of contour of integration.
Hence the goal in \cite{morosov1} is to provide exact formulas for integral discriminants  in terms of 
algebraic invariants of the form $g$ involved in the exponent; several highly non trivial examples are provided in 
\cite{morosov2,shakirov} and some of their results are further analyzed in Fuji \cite{fujii} and Stoyanovsky \cite{stoya}. 

\subsection*{Contribution} 

In the present paper we do not study non Gaussian integrals 
from the point of view of algebraic invariants as in e.g. \cite{morosov1,morosov2,shakirov}.
We are  interested in properties of the integral (\ref{def-theta}) (and variants as well) and its associated non-Gaussian integral 
$\int_{\R^n}h\exp(-g)d\x$
for PHFs that are not necessarily polynomials.
Among other things, we are also interested in methods for their evaluation (or approximation). More precisely:

(a) We first extend the action-independence property (\ref{state}) to 
PHFs of degree $0\neq d\in\R$.
Namely, given 
a measurable mapping $\phi:\R_+\to\R$, and $g,h$ PHFs of respective degree $0\neq d,p\in\R$ and such that $\int\vert h\vert \exp(-g)d\x$ is finite, 
\begin{equation}
\label{newaction}
\int_{\R^n} \phi(g(\x))h(\x)\,d\x\,=\,C(\phi,d,p)\cdot\int_{\R^n}h\exp(-g)\,d\x,\end{equation}
where the constant $C(\phi,d,p)$ depends only on $\phi,d,p$.
In particular, if the sublevel set $\{\x\,:\,g(\x)\leq 1\}$ is bounded, 
then for every nonnegative $y\in\R$,
\begin{equation}
\label{general1}
\dis\int_{\{\x\,:\: g(\x)\leq y\}}h\,d\x\,=\,
\frac{y^{(n+p)/d}}{\Gamma(1+(n+p)/d)}\dis\int_{\R^n}h\exp(-g)\,d\x,
\end{equation}
with $\Gamma$ being the standard Gamma function\footnote{The Gamma function 
$\Gamma:\R\to\R$ is defined by $a\mapsto \Gamma(a):=\int_0^\infty t^{a-1}\exp(-t)dt$, for every $a>-1$.}. 
And so the (Lebesgue) volume of the sublevel set $\{\x\,:\,g(\x)\leq y\}$ is given by:
\begin{equation}
\label{general2}
\vol(\{\x\,:\: g(\x)\leq y\})\,=\,\frac{y^{n/d}}{\Gamma(1+n/d)}\dis\int_{\R^n}\exp(-g)\,d\x.
\end{equation}
As already mentioned, when $h(\x)=1$ for all $\x$, and $g$ is a positive definite form, (\ref{general2}) was already proved in Morosov et Shakirov \cite[(75) p. 39]{morosov1}. In fact, the same arguments as in \cite{morosov1} can be adapted
to PHFs $g$ and with $h\neq1$. The same technique applies for the integral
\[I_{g_1,\ldots,g_m,h}:=\int_{\{\x\,:\,g_k(\x)\leq 1,\:k=1,\ldots,m\}}hd\x,\]
for a finite family $(g_k),h$ of PHFs with same degree $0<d\in\R$.

In addition,  we obtain the identity
\[\frac{\dis\int_{\{\x\,:\,g(\x)\leq y\}}  \exp(-g)\,h\,d\x}{\dis\int_{\R^n}\,\exp(-g)\,h\,d\x}\,=\,\frac{\dis\int_0^y\,\exp(-z)z^{(n+p)/d-1}\,dz}
{\Gamma((n+p)/d)},\]
which expresses how fast $\mu(\{\x\,:\,g(\x)\leq y\})$ converges to
$\mu(\R^n)(=\int\exp(-g)hd\x)$ when $y\to\infty$ and $\mu$ has the non Gaussian density $h\exp(-g)$ with respect to the Lebesgue measure. It is the same speed with which the truncated (Gamma) integral $\int_0^y\,\exp(-z)z^{(n+p)/d-1}\,dz$
converges to $\Gamma((n+p)/d)$.

(b) We also provide {\bf an alternative and simple proof  based on Laplace transform}
(in the spirit of Lasserre and Zeron \cite{laplace} to provide explicit expressions for certain integrals
on specific domains like e.g., a simplex or an ellipsoid). This proof permits
to interpret (\ref{first}) as duality result in the usual $(\cdot,+)$-algebra
which parallels duality in convex optimization, i.e.,  in the $(\max,+)$-algebra.
This complements the list of the many remarkable properties of homogeneous functions; for instance 
in \cite{pams} for integration, 
in \cite{jbhu-las} for optimization, and in \cite{jca} for several properties of some {\it conjugates} of homogeneous functions.

(c)  {\bf Convexity:} (\ref{general1}) also helps analyze properties of the mapping
$f_h: C_d\to \R$,
\[g\mapsto f_h(g):=\int_{\{\x\,:\,g(\x)\leq 1\}} h(\x)\,d\x,\qquad \forall\,g\in C_d,\]
where $C_d$ is the space of nonnegative PHFs of degree $0\neq d\in\R$ with bounded sublevel set $\{\x\,:\,g(\x)\leq 1\}$,
and $h$ is a given PHF.
Whenever $h$ is nonnegative, $f_h$ is strictly convex and when $g$ is a homogeneous polynomial, 
we provide an explicit expression of its gradient and Hessian on the interior of its domain.
This is the analogue in our context of a formula already provided in M\"uller et al. \cite{muller} when the integration domain is a polytope.

This convexity property is important. For instance, given $\K\subset\R^n$, it shows that 
the problem of computing a homogeneous polynomial $g$ such that $\{\x:g(\x)\leq1\}$  
contains $\K$ and has minimum volume,
is a finite-dimensional convex optimization problem (but hard to solve even though it is convex)!
This problem has received a lot of attention with an elegant solution in the case $d=2$
(that is, finding the {\em ellipsoid} of minimum volume).
Notice that here with $d>2$ the problem is still convex even though 
the sublevel set $\{\x:g(\x)\leq 1\}$ is {\it not} required to be convex.
We present  two numerical approximation schemes
based on semidefinite programming, that provide sequences of upper and lower bounds that converge to minimum volume.

(d) {\bf Polarity:} When $g$ is a convex PHF of degree $d\in\R$ with compact sublevel set  $G=\{\x :g(\x)\leq 1\}$,
the polar $G^\circ$ is a certain sublevel set of the Legendre-Fenchel conjugate $g^*$ of $g$,
which is itself a PHF of degree $q$ with $1/d+1/q=1$ and therefore we also obtain 
the volume of $G^\circ$ in the form of non 
Gaussian integral, the {\it dual} analogue of the one obtained for $G$.

(e) {\bf Numerical approximation:} Evaluating a non Gaussian integral $\int \exp(-g)$ is a very hard problem even if in some cases 
exact results can be obtained from some algebraic invariants  
of $g$, as in e.g. \cite{morosov1,morosov2,shakirov}. However, it turns out that (\ref{general1}) is indeed very useful
to approximate non Gaussian integrals from polynomials. Indeed, when $h$ is a polynomial and $g$ is a homogeneous polynomial, we show how to approximate the non Gaussian integral $\int h\exp(-g)d\x$, as closely as desired, 
by solving a hierarchy of so-called semidefinite programs. This numerical approximation  scheme 
complements the exact results mentioned above.

(f) {\bf A Gaussian-like property:} Finally we also prove a (again intriguing) Gaussian-like property of homogeneous polynomials that are sums of squares (SOS).
Recall that every homogeneous and SOS polynomial $g$ of degree $2d$ can be written 
as $\tv_d(\x)^T\si\tv_d(\x)$ for some positive semidefinite matrix $\si$ (not unique in general)
and where $\tv_d(\x)$ is the vector of all monomials of degree $d$. If we now consider the 
function 
\[\si\mapsto\qquad \theta(\si)=({\rm det}\,\si)^k\int\exp(-k\tv_d(\x)^T\si\tv_d(\x)),\]
where $k=n/(2d\ell(d))$ (with $\ell(d)={n+d-1\choose d}$), then its critical points are 
the positive semidefinite matrices $\si$ such that $\si^{-1}$ is the matrix of moments of order $2d$ 
associated with the non Gaussian density
$\exp(-k\tv_d(\x)^T\si\tv_d(\x))$ (normalized to be a probability density) ... exactly like in the Gaussian case 
(i.e. when $d=1$) where $\si^{-1}$ is a {\it covariance} matrix!
(In the latter case, the function $\theta$ is constant, hence all points $\si$ are critical).

\section{sublevel sets and non Gaussian integrals}

\subsection{Notation and definitions}

Let $\R[\x]$ be the ring of polynomials in the variables $\x=(x_1,\ldots,x_n)$ and let
$\R[\x]_d$ be the vector space of polynomials of degree at most $d$
(whose dimension is $s(d):={n+d\choose n}$).
For every $d\in\N$, let  $\N^n_d:=\{\alpha\in\N^n:\vert\alpha\vert \,(=\sum_i\alpha_i)=d\}$, 
and
let $\v_d(\x)=(\x^\alpha)$, $\alpha\in\N^n$, be the vector of monomials of the canonical basis 
$(\x^\alpha)$ of $\R[\x]_{d}$. 
Denote by  $\s_k$ the space of $k\times k$ real symmetric matrices with
scalar product $\langle \B,\C\rangle={\rm trace}\,(\B\C)$; also, the notation $\B\succeq0$ 
(resp. $\B\succ0$) stands for $\B$ is positive semidefinite (resp. positive definite).

A polynomial $f\in\R[\x]_d$ is written
\[\x\mapsto f(\x)\,=\,\sum_{\alpha\in\N^n}f_\alpha\,\x^\alpha,\]
for some vector of coefficients $\f=(f_\alpha)\in\R^{s(d)}$.

A real-valued continuous function $f:\R^n\to\R$ is homogeneous (resp. {\it positively} homogeneous) of degree $k$ ($k\in\R$)
if $f(\lambda\x)=\lambda^kf(\x)$ for all $\lambda\neq0$ (resp. for all $\lambda>0$) and all $\x\in\R$. 
For instance if $\x\mapsto f(\x)$ is homogeneous of degree $k$, then $\vert f\vert$ is positively homogeneous
of degree $k$ (but not homogeneous). Let us denote by:
\begin{itemize}
\item $C_d$ the convex cone of PHFs $g$ of degree $d\in\R$ with bounded sublevel set
$G:=\{\x\,:\,g(\x)\leq 1\}$. Notice that necessarily $d\neq0$ and any $g\in C_d$ is nonnegative. Indeed, 
if $d=0$ then $\x\in G$ implies $\lambda\x\in G$ for every $\lambda>0$, in contradiction
with boundedness of $G$. Similarly, assume that $g(\x_0)<0$ for some 
some $\x_0$ (hence $\x_0\in G$). As $g(\lambda\x_0)=\lambda^dg(\x_0)<0$ for every $\lambda>0$,
$\lambda\x_0\in G$ for every $\lambda>0$, again in contradiction with boundedness of $G$.
\item $\P[\x]_d\subset \R[\x]_d$, $d\in\N$, the convex cone of homogeneous polynomials of degree $d$ with compact 
sublevel set $\{\x\,:\,g(\x)\leq 1\}$.
Hence, $\P[\x]_d\subset C_d$ whenever $d\in\N$.
\end{itemize}

Remember that a convex function is {\it proper} if $f(\x)>-\infty$ for all $\x$ and $f(\x)<+\infty$ for some $\x$. 
A proper convex function is closed if it is lower semicontinuous; see Rockafellar \cite{rockafellar}.
 
In the sequel we will need the following result which extends to PHFs of degree $d\in\R$,
one already used in \cite{morosov1} for forms 
of degree $d\in\N$ (and also  used
in Hiriart-Urruty \cite{jbhu} for the volume of the polar of a convex set).

\begin{lemma}
\label{lemma0}
Let $d,p\in\R$ with $d\neq0$, and let $\phi:\R_+\to\R_+$ be a measurable function such that 
\[\int_0^\infty t^{n+p-1}\phi(t^d)\,dt<\infty,\] 
and let $g,h$ be nonnegative PHFs of degree $0\neq d\in\R$ and $p\in\R$ respectively. The
functional $\dis\int_{\R^n} \phi(g(\x))h(\x)\,d\x$ is finite if and only if $\dis\int_{\R^n}h\exp(-g)d\x<+\infty$, in which case:
\begin{equation}
\label{expo}
\int_{\R^n} \phi(g(\x))h(\x)\,d\x
\,=\,\underbrace{\frac{d\dis\int_0^\infty t^{n+p-1}\phi(t^d)\,dt}{\Gamma((n+p)/d)}}_{{\rm cte}(n,p,d,\phi)}\cdot\int_{\R^n}h\,\exp(-g)\,d\x.
\end{equation}
\end{lemma}

\begin{proof}
With $\z=(z_1,\ldots,z_{n-1})$, do the change of variable $x_1=t$, $x_2=tz_1,\dots,x_n=tz_{n-1})$ so that one may 
decompose $\int_{\R^n}\phi(g(\x))h(\x)d\x$ into the sum
\begin{eqnarray*}
&&\int_{\R_+\times\R^{n-1}}t^{n+p-1}\phi(t^dg(1,\z))h(1,\z)
\,dt\,d\z\\
&+&\int_{\R_+\times\R^{n-1}} t^{n+p-1}\phi(t^dg(-1,-\z))h(-1,-\z)
\,dt\,d\z,\\
&=&\int_{\R^{n-1}}\left(\int_0^\infty t^{n+p-1}\phi(t^dg(1,\z))\,dt\right)h(1,\z)\,d\z\\
&+&\int_{\R^{n-1}}\left(\int_0^\infty t^{n+p-1}\phi(t^dg(-1,-\z))\,dt\right)h(-1,-\z)\,d\z,\\
\end{eqnarray*}
where the last two integrals are obtained from the sum of the previous two by using Tonelli's Theorem
(since $\phi$ and $h$ are nonnegative functions) and whether or not the integrals are finite; see e.g. Dunford and Schwartz
\cite[Theorem 14, p. 194]{dunford}.

Next, with the change of variable $u=t\,g(1,\z)^{1/d}$ and $u=t\,g(-1,-\z)^{1/d}$
\begin{equation}
\label{aux77}
\int_{\R^n} \phi(g(\x))\,h(\x)d\x\,=\,\left(\int_{\R_+} u^{n+p-1}\phi(u^d)\,du\right)\cdot A(g,h),\end{equation}
with
\begin{equation}
\label{A}
A(g,h)=\int_{\R^{n-1}}\left(\frac{h(1,\z)}{g(1,\z)^{(n+p)/d}}+
\frac{h(-1,-\z)}{g(-1,-\z)^{(n+p)/d}}\right)\,d\z.\end{equation}
In particular, the choice $t\mapsto \phi(t)=\exp(-t)$ yields:
\begin{equation}
\label{AA}
A(g,h)\,=\,\frac{d}{\Gamma((n+p)/d)}\int_{\R^n} h\exp(-g)\,d\x,\end{equation}
and so $A(g,h)$ is finite if and ony if $\int h\exp(-g)d\x$ is finite. 
Substituting $A(g,h)$ in (\ref{aux77}) with (\ref{AA}) yields (\ref{expo}).
\end{proof}

\subsection{Integration on sublevel sets  and non Gaussian integrals}

As a consequence of Lemma \ref{lemma0} we obtain the following result.
\begin{thm}
\label{thmain}
Let $g,h$ be PHFs of degree $0 \neq d\in\R$ and $p\in\R$ respectively,
and such that $g\in C_d$. Then for every $y\in [0,\infty)$:
\begin{equation}
\label{b1}
\vol(\{\x\,:\,g(\x)\leq y\})\,=\,\frac{y^{n/d}}{\Gamma(1+n/d)}\int_{\R^n}\exp(-g)\,d\x,
\end{equation}
and
\begin{equation}
\label{b2}
\int_{\{\x\,:\,g(\x)\leq y\}}h\,d\x\,=\,\frac{y^{(n+p)/d}}{\Gamma(1+(n+p)/d)}\int_{\R^n}h\,\exp(-g)\,d\x,
\end{equation}
whenever $\int_{\R^n}\vert h\vert \exp(-g)d\x$ (or $\int_{\{\x\,:\,g(\x)\leq 1\}}\vert h\vert \,d\x$) is finite.
And (\ref{b2}) also holds if 
$h$ is replaced with $\vert h\vert$, or with $\max[0,h]$, or with $\max[0,-h]$.
\end{thm}
\begin{proof}
Write $h=h^+-h^-$ with $h^+(\x):=\max[0,h(\x)]$, $h^-(\x):=\max[0,-h(\x)]$ for all $\x$.
Observe that $\vert h\vert$, $h^+$ and $h^-$ are nonnegative and positively homogeneous of degree $p$, and
$\vert h\vert=h^++h^-$, $h=h^+-h^-$.

With $A$ as in (\ref{aux77}), using  Lemma \ref{lemma0} with $\phi(t)=\exp(-t)$, yields 
\begin{eqnarray*}
\int_{\R^n}\vert h\vert\,\exp(-g)\,d\x&=&\frac{\Gamma(1+(n+p)/d)}{n+p}\cdot A(g,\vert h\vert)\\
\end{eqnarray*}
where all integrals are finite since the right-hand-side is finite. Similarly, using  Lemma \ref{lemma0} with $\phi(t)=1_{[0,1]}(t)$, yields
\[\int_{\{\x:g(\x)\leq1\}}\vert h\vert \,d\x\,=\,\frac{1}{n+p}\cdot A(g,\vert h\vert),\]
and therefore,
\[\int_{\{\x\,:\,g(\x)\leq 1\}}\vert h\vert \,d\x\,=\,\frac{1}{\Gamma(1+(n+p)/d)}\int_{\R^n}\vert h\vert \,\exp(-g)\,d\x.\]
Next, by homogeneity,
\[\int_{\{\x\,:\,g(\x)\leq y\}}\vert h\vert \,d\x\,=\,y^{(n+p)/d}\int_{\{\x\,:\,g(\x)\leq 1\}}\vert h\vert \,d\x,\]
and so
\begin{equation}
\label{aux66}
\int_{\{\x\,:\,g(\x)\leq y\}}\vert h\vert \,d\x\,=\,\frac{y^{(n+p)/d}}{\Gamma(1+(n+p)/d)}\int_{\R^n}\vert h\vert \,\exp(-g)\,d\x.
\end{equation}
Obviously, with same arguments, (\ref{aux66}) also holds if we replace $\vert h\vert$ with $h^+$ and $h^-$
and finiteness follows because 
\[0\leq\,\int_{\{\x\,:\,g(\x)\leq y\}} h^+ \,d\x\,\leq\,\int_{\{\x\,:\,g(\x)\leq y\}}\vert h\vert \,d\x,\]
and similarly for $h^-$. Hence (\ref{aux66}) with $h$ in lieu of $\vert h\vert$ (i.e. (\ref{b2}))
follows by additivity since $h=h^+-h^-$.

Finally, from what precedes, finiteness of
$\int_{\{\x:g(\x)\leq 1\}}\vert h\vert d\x$ is equivalent to finiteness of
$\int_{\R^n}\vert h\vert \exp(-g)d\x$.
\end{proof}

When $y=1$, $g$ is a positive definite form with $d\in\N$, and $h$ is identical to the constant function $1$,
(\ref{b1}) is already proved in Morosov and Shakirov \cite[(75) p. 39]{morosov1}.

In particular, when $g$ is the quadratic form 
$\x\mapsto g_2(\x):=\frac{1}{2}\x^T\Q\x$ for some real symmetric positive definite matrix $\Q$, one retrieves that
the volume of the ellipsoid $\xi(y):=\{\x\,:\,g_2(\x)\leq y\}$ is simply related to the determinant of
$\Q$ by the formula
\begin{equation}
\label{det}
\vol(\xi(y))\,=\,\frac{y^{n/2}}{\Gamma(1+n/2)}\underbrace{\dis\int_{\R^n}\exp(-\frac{1}{2}\x^T\Q\x)\,d\x}_{=(2\pi)^{n/2}/\sqrt{\det\Q}}
\end{equation}
So Theorem \ref{thmain} states that the volume of the sublevel set is simply related to
the integral of $\exp(-g(\x))$ over the entire domain $\R^n$ which happens to be simply related to
the determinant of $\Q$ when $g$ is the quadratic form $\x^T \Q\x$. 
One goal of the theory of integral discriminants is precisely to express 
$\int \exp(-g)$ in terms of invariants of $g$ when $g$ is a form. See e.g. Morosov and Shakirov \cite{morosov1,morosov2}
and Shakirov \cite{shakirov}.

\subsection*{Intersection of sublevel sets}

Suppose that with $h$ being positively homogeneous of degree $p$,
one wishes to compute the integral 
$\int_{\Omega}h(\x)\,d\x$ where 
\[\Omega:=\{\x\in\R^n\::\: g_k(\x)\leq z_k,\quad k=1,\ldots ,m\},\]
for some $m$ PHFs $g_1,\ldots,g_m$ of degree $0\neq d\in\R$, and some (strictly) positive vector
$\z\in\R^m$. Equivalently,
\[\Omega:=\{\x\in\R^n\::\: \tilde{g}_k(\x;\z)\leq 1,\quad k=1,\ldots ,m\},\]
for the functions $\x\mapsto \tilde{g}_k(\x;\z):=g_k(z_k^{-1/d}\x)$, $k=1,\ldots,m$,
which are also PHFs of degree $d\in\R$. Hence with no loss of generality,
one may restrict to sets of the form
\begin{equation}
\label{newset}
\Omega(y)\,:=\,\{\x\::\:g_k(\x)\,\leq\,y,\quad k=1,\ldots,m\}\end{equation}
for some positive scalar $y\in\R$, and 
PHFs $g_1,\ldots,g_m$ of same degree $d\in\R$.

Notice that 
$\x\mapsto \psi(\x):=\max[g_1(\x),\ldots,g_m(\x)]$
is a PHF of degree $d$.

\begin{corollary}
\label{coro4}
Let $h$ be a PHF of degree $p\in\R$, let
$g_1,\ldots,g_m$ be PHFs of degree $0\neq d\in\R$, and
assume that the set $\{\x\,:\, g_k(\x)\leq 1,\: k=1,\ldots,m\}$ is bounded and for every $y>0$, let
$\Omega(y)$ be as in (\ref{newset}). Then:
\begin{equation}
\label{coro4-2}
\vol (\Omega(y))\,=\,\frac{y^{n/d}}{\Gamma(1+n/d)}\int_{\R^n}\exp(-\psi)\,d\x,
\end{equation}
and
\begin{equation}
\label{coro4-1}
\int_{\Omega(y)}h(\x)\,d\x\,=\,\frac{y^{(n+p)/d}}{\Gamma(1+(n+p)/d)}
\int_{\R^n}h\,\exp(-\psi)\,d\x,
\end{equation}
whenever $\int_{\R^n}\vert h\vert \,\exp(-\psi)\,d\x$.
\end{corollary}
\begin{proof}
Notice that  $\x\mapsto \psi(\x):=\max[g_1(\x),\ldots,g_m(\x)]$ is also a PHF of degree $0\neq d\in\R$,
and $\Omega(y)=\{\x\,:\,\psi(\x)\leq y\}$. And so if $\Omega(y)$ is bounded then 
$\psi\in C_d$. Hence (\ref{coro4-1})-(\ref{coro4-2}) is just (\ref{b1})-(\ref{b2}) with
$h$ and $\psi$ in lieu of $h$ and $g$.
\end{proof}

\subsection{An alternative proof with a duality interpretation}

Next, we present an alternative proof of Theorem \ref{thmain} that uses Laplace transform techniques
and provides an interpretation of the result in an appropriate {\it duality} framework.

Suppose that $g,h\in C_d$.
Since $g$ is nonnegative, the function $I_{g,h}$ vanishes on $(-\infty,0]$. 
Its Laplace transform $\mathcal{L}_{I_{g,h}}: \mathbb{C}\to\R$ is the function
\[\lambda\mapsto \mathcal{L}_{I_{g,h}}(\lambda):=\int_0^\infty\exp(-\lambda y) I_{g,h}(y)\,dy,\]
and observe that 
\begin{eqnarray*}
\mathcal{L}_{I_{g,h}}(\lambda)&=&\int_0^\infty\exp(-\lambda y)\left(\dis\int_{\{\x: g(\x)\leq y\}}hd\x\right)\,dy\\
&=&\dis\int_{\R^n}h(\x)\left( \int_{g(\x)}^\infty \exp(-\lambda y)dy\right)\,d\x\quad\mbox{[by Fubini's Theorem]}\\
&=&\frac{1}{\lambda}\dis\int_{\R^n}\,h(\x)\exp(-\lambda g(\x))\,d\x\\
&=&
\frac{\lambda^{-p/d}}{\lambda}\dis\int_{\R^n}\,h(\lambda^{1/d}\x)\exp(-g(\lambda^{1/d}\x))\,d\x\quad\mbox{[by homogeneity]}\\
&=&\frac{1}{\lambda^{1+(n+p)/d}}\dis\int_{\R^n}h(\z)\exp(-g(\z))\,d\z\quad\mbox{[by $\lambda^{-1/d}\x\to\z$]}\\
&=&\frac{\dis\int_{\R^n}h(\z)\exp(-g(\z))\,d\z}{\Gamma(1+(n+p)/d)}\:\mathcal{L}_{y^{(n+p)/d}}(\lambda).
\end{eqnarray*}
And so, by uniqueness of 
the Laplace transform,
\[I_{g,h}(y)=\frac{y^{(n+p)/d}}{\Gamma(1+(n+p)/d)}\dis\int_{\R^n}\,h(\x)\exp(-g(\x))\,d\x,\]
which is the desired result. So the above expression of $I_{g,h}(y)$ is obtained by
``inverting" the Laplace transform $\mathcal{L}_{I_{g,h}}$ at the point $y$, 
which in fact, as we next see, is solving a ``dual" problem.

For analogy purposes, consider, the optimization problem 
\[ \rho_{g,h}(y)\,=\,\sup_\x\{ h(\x)\::\:g(\x)\leq y\},\]
where $y$ is fixed, $h$ and $-g$ are concave and $h$ is nonnegative. 
Equivalently, $\rho_{g,h}(y)=\exp(\theta_{g,h}(y))$ where $y\mapsto \theta_{g,h}(y)$ is the {\it optimal value function} of the optimization problem 
\[\P_y:\quad \theta_{g,h}(y)\,=\,\sup_\x\:\{ \ln h(\x)\::\:g(\x)\leq y\,\}.\]
Associated with $\P_y$ is the {\it dual} problem $\P^*_y:\:\inf_\lambda\,\{\G(\lambda)\::\:\lambda \geq0\}$,
where $G:\R_+\to\R$ is the function $\lambda\mapsto G(\lambda):=\sup_{\x}\:\{\,\ln h(\x)+\lambda (y-g(\x))\}$.
Observe that:
\begin{eqnarray*}
\lambda\mapsto G(\lambda)&=&\lambda y+\sup_{\x}\:\{\,\ln h(\x)-\lambda g(\x)\}\\
&=&\lambda y+\sup_{\z}\:\{\,\ln (\lambda^{-p/d}h(\z))-g(\z)\}\quad\mbox{[via $\z=\lambda^{1/d}\x$]}\\
&=&\lambda y+-\frac{p}{d}\ln \lambda +\sup_{\z}\{\ln h(\z)-g(\z)\}.
\end{eqnarray*}
And so the dual problem $\P^*_y$ reads
\begin{eqnarray*}
\P^*_y:\quad\gamma&=&\sup_{\z}\{\ln h(\z)-g(\z)\}+\inf_{\lambda\geq0}\:\{\lambda y-\frac{p}{d}\ln \lambda\,\}\\
&=&\ln y^{p/d} +\ln\left(\sup_{\z}\{h(\z)\exp(-g(\z))\}\right)+\frac{p}{d}(1-\ln \frac{p}{d})
\end{eqnarray*}
 In particular, if $h$ is log-concave and $g$ is convex, by a standard argument of convex optimization,
 $\gamma=\theta_{g,h}(y)\,(=\ln\rho_{g,h}(y))$. And therefore,
 \[\rho_{g,h}(y)\,=\,\exp(\theta_{g,h}(y))\,=\,y^{p/d}\, \frac{\exp(p/d)}{(p/d)^{p/d}}\,\sup_\x \{\,h(\x)\exp(-g(\x))\,\} \]
to compare with
\[I_{g,h}(y)=y^{(n+p)/d}\,\frac{1}{\Gamma(1+(n+p)/d)}\,\dis\int_{\R^n}\,h(\x)\exp(-g(\x))\,d\x.\]

Alternatively, the Legendre-Fenchel transform of $\theta_{g,h}(y)$ (for the concave version) is the function
\begin{eqnarray*}
\theta_{g,h}^*(\lambda)&=&\inf_y \{\lambda y-\theta_{g,h}(y)\}\\
&=&\inf_\lambda \{\lambda y-\sup_x \{\ln h(\x)\::\:g(\x)\leq y\}\,\}\\
&=&\inf_\x\,\{\lambda g(\x)-\ln h(\x)\}\,=\,-\ln\left(\sup_\x\{h(\x)\,\exp(-\lambda g(\x))\}\right)
\end{eqnarray*}
and so when $h$ is log-concave and $g$ is convex, $\theta_{g,h}(y)=(\theta^*_{g,h})^*(y)$, so that
\begin{eqnarray*}
\theta_{g,h}(y)
&=&\inf_\lambda \{\lambda y-\theta_{g,h}^*(\lambda)\}\\
&=&\inf_\lambda\,\{\lambda y+\ln\left(\sup_\x\{ h(\x)\,\exp(-\lambda g(\x))\}\right)\,\}\\
&=&\ln y^{p/d} +\ln\left(\sup_{\z}\{h(\z)\exp(-g(\z))\}\right)+\frac{p}{d}(1-\ln \frac{p}{d}).
\end{eqnarray*}

In summary, to the Laplace transform step 
\begin{equation}
\label{laplace}
\mathcal{L}_{I_{g,h}}(\lambda)\,=\,\int_{\R^n} \exp(-\lambda y)I_{g,h}(y)\,dy\end{equation}
in the usual $(\cdot,+)$-algebra, corresponds to the Legendre Fenchel transform 
\begin{eqnarray*}
\theta_{g,h}^*(\lambda)=\inf_y \{\lambda y-\theta_{g,h}(y)\}&=&-\sup_y \{-\lambda y+\theta_{g,h}(y)\}\\
&=&-\ln\sup_y \{\exp(-\lambda y)\rho_{g,h}(y)\},\end{eqnarray*}
where to emphasize the analogy, the latter term can be written 
\[\ln\left( ``\int" \exp(-\lambda y) \,\rho_{g,h}(y)\,\right),\]
i.e., the $``\sup"$ operator is integration $``\int"$  in the $(\max,+)$-algebra. And with this convention
\[\exp(\theta_{g,h}^*(\lambda))\,\simeq\,``\int"\exp(-\lambda y)\,\rho_{g,h}(y).\]
(Compare with (\ref{laplace}).)
Similarly, the Laplace inverse transform step 
\begin{equation}
\label{laplaceinv}
I_{g,h}(y)=\int_{\omega-i\infty}^{\omega+i\infty} \exp(\lambda y)\,\mathcal{L}_{I_{g,h}}(\lambda)\,d\lambda\end{equation}
(where $\omega\in\mathbb{C}$ is on the right of all singularities of $\mathcal{L}_{I_{g,h}}$)
and which yields
\[I_{g,h}(y)=\frac{y^{(n+p)/d}}{\Gamma(1+(n+p)/d)}\int_{\R^n}h\exp(-g)d\x,\]
is solving the ``dual" and corresponds to the Legendre-Fenchel transform (which is involutive) applied to $\theta_{g,h}^*$
\[\exp(\theta_{g,h}(y))\,=\,\exp(\inf_\lambda\{\lambda y-\theta_{g,h}^*(\lambda)\})\,\simeq\,-``\int"\exp(-\lambda y)\,\exp(\theta_{g,h}^*(\lambda))\,d\lambda\]
(compare with (\ref{laplaceinv})), and which yields 
\[\exp(\theta_{g,h}(y))=y^{p/d}\, \frac{\exp(p/d)}{(p/d)^{p/d}}\:``\int" \,h(\x)\exp(-g(\x)).\]
A rigorous analysis of the links between integration and linear optimization on a polytope
has been already investigated in \cite{lasserrebook}.

\subsection{Approximating non Gaussian integrals}

As we have already mentioned, in some cases the non Gaussian integral 
can be computed explicitly in terms of some on algebraic invariants of $g$; see e.g.
Morosov and Shakirov \cite{morosov1} and Shakirov \cite{shakirov}.
But so far there is no general formula and therefore an alternative is
to seak for a numerical scheme for its evaluation, or at least, its approximation. 

For this purpose, we next show that Theorem \ref{thmain} is helpful as it provides a means to compute any moment of the measure $d\mu=\exp(-g)d\x$ on $\R^n$ by
computing the same moment but now of the Lebesgue measure on the sublevel set $\{\x: g(\x)\leq 1\}$. Indeed, for every
 $\alpha\in\N^n$, letting $\x\mapsto h(\x):=\x^\alpha$ in Theorem \ref{thmain}, yields
\begin{equation}
\label{moment}
\int_{\R^n} \x^\alpha\exp(-g(\x))\,d\x\,=\,\Gamma(1+(n+\vert\alpha\vert)/d)\int_{\{\x\,:\,g(\x)\leq1\}}\x^\alpha\,d\x,
\end{equation}
where $\vert\alpha\vert=\sum_i\alpha_i$. Have we made any progress with this equivalence?

The answer is yes. If $g$ is a (non necessarily homogeneous) polynomial, it turns out that every moment of the Lebesgue measure on the sublevel set
$\{\x\,:\,g(\x)\leq1\}$ can be approximated as closely as desired by solving a hierarchy of semidefinite programs\footnote{A semidefinite program is a finite-dimensional convex conic optimization problem, that up to arbitrary (fixed) precision,
can be solved efficiently, i.e., in time polynomial in the input size of the problem. For more details the interested reader is referred to e.g. \cite{wolko}.}
as described in Henrion et al. \cite{sirev}.
In fact, for every $\alpha\in\N^n$ fixed, the moment
\[z_\alpha\,:=\,\int_{\{\x\,:\,g(\x)\leq1\}}\x^\alpha\,d\x\]
can be approximated to arbitrary precision $\epsilon>0$ fixed in advance,
by solving two sequences of semidefinite programs, one which provides
a monotone non decreasing sequence of upper bounds $u_k$, $k\in\N$, while the other provides
a monotone non increasing sequence of lower bounds $\ell_k$, $k\in\N$. The procedure stops whenever 
$u_k-\ell_k<\epsilon$, in which case one may set 
\[z_\alpha\approx \tilde{z}_\alpha\,:=\,(u_k+\ell_k)/2.\]
This requires to solve two sequences of semidefinite programs for each $\alpha\in\N^n$. 
In fact, if one is ready to relax the monotonicity property of the upper and lower bounds $\{u_k,\ell_k\}$, it is enough to solve a single sequence
of semidefinite programs, e.g., the one defined to approximate the mass $z_0$. 
Then if $d\in\N$ and $\epsilon >0$ are fixed, and $k$ is large enough,
 not only $\vert u_k-z_0\vert<\epsilon$ but also from the solution of the semidefinite program at step $k$ one obtains
 scalars $\tilde{z}_\alpha$ such that  $\vert \tilde{z}_\alpha-z_\alpha\vert<\epsilon$, for all $\alpha\in\N^n$ such that 
 $\vert\alpha\vert <d$. However, in contrast to the case of upper and lower bounds, 
 there is no simple stopping criterion to guarantee the $\epsilon$-approximation.
For more details, the interested reader is referred to Henrion et al. \cite{sirev}. And therefore,
once the $\tilde{z}_\alpha$ have been computed, using Theorem \ref{thmain}, we obtain:
\[\left\vert 
\int_{\R^n} \x^\alpha\exp(-g(\x))\,d\x\,-\,\tilde{z}_\alpha\,\Gamma(1+(n+\vert\alpha\vert)/d)
\right\vert\,<\,\epsilon,\quad\forall\alpha\in\N^n_d,\]
which provides an $\epsilon$-approximation guarantee for the non Gaussian integral $\int_{\R^n}\exp(-g)d\x$ (and more generally for
the integral $\int_{\R^n} h\exp(-g)d\x$ whenever $h$ is any polynomial).

\subsection{Sensitivity analysis and convexity}

Recall that when $d\in\N$, $\P[\x]_{d}\subset C_d$ is the convex cone of nonnegative and 
homogeneous polynomials of degree $d$, with compact sublevel set
$\{\x\,:\,g(\x)\leq 1\}$.
Formula (\ref{b2}) of Theorem \ref{thmain} allows us
to provide insights into the function $f:C_d\to\R$, defined by:
\begin{equation}
\label{def-f}
g\mapsto f_h(g)\,:=\,\int_{\{\x\,:\,g(\x)\leq 1\}} h(\x)\,d\x,\qquad g\in\,C_d,
\end{equation}
where $h$ is a PHF. 
In particular when one wishes to 
see how $f_h$ changes when some coefficient of $g\in \P[\x]_d$ varies.
Notice that the restriction of $f_h$ to $\P[\x]_d$
may be seen as a function $f_h:\R^{\ell(d)}\to\R$ of the coefficient vector
of the polynomial $g\in\P[\x]_d$, where $\ell(d)={n+d-1\choose d}$.

Before proceeding further we need the following result. Recall that the support $\supp\mu$
of a Borel measure on $\R^n$ is the smallest closed set $A$ such that $\mu(\R^n\setminus A)=0$.
Let $C^0_d:=\{ g\in C_d: \mbox{ $g$ is continuous}\}\subset C_d$.

\begin{lemma}
\label{lemma-aux2}
Let $\mu$ be a non trivial $\sigma$-finite Borel measure on $\R^n$ and
let $\Theta_\mu\subset\R[\x]_d$ be the convex cone of 
polynomials $g$ of degree at most $d$ such that $\int\exp(-g)d\mu <\infty$.
Then:

{\rm (a)} With $d\in\R$, the function $f:\,C_d\to\R$, with $g\mapsto f(g):=\int\exp(-g)d\mu$, is convex
(and strictly convex on $C^0_d$ if $\supp\mu=\R^n$).
 
 {\rm (b)} If $\mu(O)>0$ for some open set $O\subset\R^n$ and if $d\in\N$, then $f$ is strictly convex and twice differentiable
on ${\rm int}(\Theta_\mu)$, with:
\begin{eqnarray}
\label{aux2-1}
\frac{\partial f(g)}{\partial g_\alpha}&=&\int \x^\alpha\,\exp(-g)\,d\mu,\qquad\forall\alpha,\,\:\vert\alpha\vert\leq d.\\
\label{aux2-2}
\frac{\partial^2 f(g)}{\partial g_\alpha\partial g_\beta}&=&\int \x^{\alpha+\beta}\,\exp(-g)\,d\mu,\qquad\forall\alpha,\beta,\,\:\vert\alpha\vert,\:\vert\beta\vert\leq d.
\end{eqnarray}
\end{lemma}
\begin{proof}
(a) Observe that $f$ is nonnegative. 
Let $\alpha\in [0,1]$ and let $g,q\in C_d$.
To prove $f(\alpha g+(1-\alpha) q)\leq\alpha f(g)+(1-\alpha)f(q)$, 
we only need consider the case where $f(g),f(q)<+\infty$,  for which we have
\[f(\alpha g+(1-\alpha)q)\,=\,\int \exp(-\alpha g-(1-\alpha)q)\,d\mu.\]
By convexity of $u\mapsto \exp(-u)$,
\begin{eqnarray*}
f(\alpha g+(1-\alpha)q)&\leq&\dis\int [\,\alpha\exp(-g)+(1-\alpha)\exp(-q)\,]\,d\mu\\
&=&\alpha f(g)+(1-\alpha)f(q),
\end{eqnarray*}
and so $f$ is convex. Now, in view of the strict convexity of $u\mapsto \exp(-u)$, 
equality may occur only if $g(\x)=q(\x)$, $\mu$-almost everywhere. If $g,q\in C^0_d$,
the set $\Delta:=\{\x :g(\x)-q(\x)=0\}$ is closed and so if $\Delta\neq\R^n$ then $\mu(\Delta)<\mu(\R^n)$ because
$\supp\mu=\R^n$. Therefore, equality occurs only if $g=q$ so that $f$ is strictly convex on $C^0_d$.\\

(b) Next, if $d\in\N$ and $g\in{\rm int}(\Theta_\mu)$, write $g$ in the canonical basis as
 $g(\x)=\sum_\alpha g_\alpha\x^\alpha$. For every
$\alpha\in\N^n_d$, let $e_\alpha=(e_\alpha(\beta))\in\R^{s(d)}$ be such that
$e_\alpha(\beta)=\delta_{\beta=\alpha}$ (with $\delta$ being the Kronecker symbol).
 Then for every $t\geq0$,
\[\frac{f(g+te_\alpha)-f(g)}{t}\,=\,\int \exp(-g)\,\left(\underbrace{\frac{\exp(-t\x^\alpha)-1}{t}}_{\psi(t,\x)}\right)\,d\mu(\x)\]
Notice that for every $\x$, by convexity of the function $t\mapsto \exp(-t\x^\alpha)$, 
\[\lim_{t\downarrow0}\psi(t,\x)\,=\,\inf_{t\geq0}\psi(t,\x)\,=\,\exp(-t\x^\alpha)'_{\vert t=0}\,=\,-\x^\alpha,\]
because for every $\x$, the function $t\mapsto \psi(t,\x)$ is nondecreasing; see e.g. Rockafellar \cite[Theorem 23.1]{rockafellar}.
Hence, the one-sided directional derivative $f'(g;\e_\alpha)$ in the direction $e_\alpha$ satisfies
\begin{eqnarray*}
f'(g;e_\alpha)&=&\lim_{t\downarrow 0}\frac{f(g+te_\alpha)-f(g)}{t}\,=\,
\lim_{t\downarrow 0}\int \exp(-g)\,\psi(t,\x)\,d\mu(\x)\\
&=&\int \exp(-g)\,\lim_{t\downarrow 0}\psi(t,\x)\,d\mu(\x)\,=\,\int-\x^\alpha \exp(-g)\,d\mu(\x),
\end{eqnarray*}
where the third equality follows from the Extended Monotone Convergence Theorem \cite[1.6.7]{ash}.
Indeed for all $t<t_0$ with $t_0$ sufficiently small, the function $\psi(t,\cdot)$ is bounded above
by $\psi(t_0,\cdot)$ and $\int \exp(-g)\psi(t_0,\x)d\mu<\infty$. Similarly, for every $t>0$
\[\frac{f(g-te_\alpha)-f(g)}{t}\,=\,
\int\exp(-g)\,\underbrace{\frac{\exp(t\x^\alpha)-1}{t}}_{\xi(t,\x)}\,d\mu(\x),\]
and by convexity of the function $t\mapsto \exp(t\x^\alpha)$
\[\lim_{t\downarrow0}\xi(t,\x)\,=\,\inf_{t\geq0}\xi(t,\x)\,=\,\exp(t\x^\alpha)'_{\vert t=0}\,=\,\x^\alpha.\]
Therefore, with exactly same arguments as before,
\begin{eqnarray*}
f'(g;-e_\alpha)&=&\lim_{t\downarrow 0}\frac{f(g-te_\alpha)-f(g)}{t}\\
&=&\int\x^\alpha\exp(-g)\,d\mu(\x)=-f'(g;e_\alpha),\end{eqnarray*}
and so 
\[\frac{\partial f(g)}{\partial g_\alpha}\,=\,-\int_{\R^n}\x^\alpha\,\exp(-g)\,d\mu(\x),\]
for every $\alpha$ with $\vert\alpha\vert\leq d$, which yields (\ref{aux2-1}). 
Similar arguments can used for the Hessian $\nabla^2f(g)$ which yields (\ref{aux2-2}).

So the Hessian $\nabla^2f(g)$ is the matrix $\M_d\in\s_{\ell(d)}$ whose rows and columns are indexed 
in the set $\Gamma_d:=\{\alpha\in\N^n:\vert\alpha\vert=d\}$ and with entries
\[\M_d(\alpha,\beta)\,=\,\int_{\R^n}\x^{\alpha+\beta}\,\underbrace{\exp(-g)\,d\mu}_{d\nu},\qquad \alpha,\beta\in\Gamma_d,\]
i.e., $\M_d$ is the matrix of $2d$-moments of the finite Borel measure $\nu$.
Let $\h\in\R^{\ell(d)}$ be the coefficient vector of a non trivial and arbitrary homogeneous polynomial $h\in\R[\x]_d$. Then
\[\langle\h,\M_d\h\rangle\:\left(\,=\,\int_{\R^n}h(\x)^2\,d\nu(\x)\right)\,>\,0\]
because $\mu(O)>0$ (hence $\nu(O)>0$) on some open set $O\subset\R^n$. Therefore $\nabla^2f(g)\succ0$, which in turn implies that 
$f$ is strictly convex  on ${\rm int}(\Theta_\mu)$.
\end{proof}

\begin{corollary}
\label{coro-convex}
Let $h$ be a PHF of degree $p\in\R$ and with $0\neq d\in\R$, consider the function
$f_h: C_d\to\R$ defined by:
\begin{equation}
\label{coro1-0}
g\,\mapsto\quad f_h(g)\,:=\,\dis\int_{\{\x:g(\x)\leq 1\}}h(\x)\,d\x,\qquad \forall g\in C_d.
\end{equation}
The function $f_h$ is a PHF of degree $-(n+p)/d$ and convex whenever $h$ is nonnegative
(and strictly convex if $h>0$ on $\R^n\setminus\{0\}$).
In addition, if  $d\in\N$ and $h$ is continuous with $\int \vert h\vert \exp(-g)d\x<\infty$, then:\\

{\rm (a)} $f_h$ is twice differentiable on ${\rm int}(\P[\x]_d)$, and for every 
$\alpha,\beta\in\N^n_{d}$:
\begin{eqnarray}
\label{coro1-1}
\frac{\partial f_h(g)}{\partial g_\alpha}&=&\frac{-1}{\Gamma(1+(n+p)/d)}
\dis\int_{\R^n} \x^\alpha \,h(\x)\exp(-g(\x))\,d\x\\
\label{coro1-2}
&=&\frac{-\Gamma(2+(n+p)/d)}{\Gamma(1+(n+p)/d)}\dis\int_{\{\x\,:\,g(\x)\leq 1\}} \x^\alpha \,h(\x)\,d\x.\\
\label{coro1-3}
\frac{\partial^2 f_h(g)}{\partial g_\alpha\partial g_\beta}&=&
\frac{1}{\Gamma(1+(n+p)/d)}\dis\int_{\R^n} \x^{\alpha+\beta} \,h(\x)\exp(-g(\x))\,d\x.
\end{eqnarray}
(b) If $h$ is non trivial and nonnegative, then $f_h$ is strictly convex
on ${\rm int}(\P[\x]_d)$, and its Hessian $\nabla^2f_h(g)$ is the matrix of 
$2d$-moments of the measure
\[\nu(B)=\frac{1}{\Gamma(1+(n+p)/d)}\dis\int_{B} h\,\exp(-g)\,d\x,\qquad B\in\mathcal{B}.\]
\end{corollary}
\begin{proof}
With $\lambda>0$ and $g\in C_d$,
\[f_h(\lambda g)=\int_{\{\x:\lambda g(\x)\leq1\}}h\,d\x=\int_{\{\x:g(\lambda^{1/d}x)\leq1\}}h\,d\x,\]
and so, doing the change of variable $\z=\lambda^{1/d}\x$, one obtains
$f_h(\lambda g)=\lambda^{-(n+p)/d}f_h(g)$, i.e.,
$f_h$ is a PHF of degree $-(n+p)/d$.

Next, first consider the case where $h$ is nonnegative, and let
$\mu$ be the $\sigma$-finite measure defined by $\mu(B):=\dis\int_Bh(\x)d\x$ for every Borel set $B$ of $\R^n$. 
By Theorem \ref{thmain}, whenever $g\in C_d$ and $f_h(g)$ is finite,
\[f_h(g)\,=\,\frac{1}{\Gamma(1+(n+p)/d)}\dis\int_{\R^n}\exp(-g)\,d\mu.\]
Hence by Lemma \ref{lemma-aux2}, $f_h$ is convex (and strictly convex if $h>0$ on $\R^n\setminus\{0\}$).

In addition, if $d\in\N$ and $h$ is continuous, $f_{h}(g)<\infty$ if $g\in{\rm int}(\P[\x]_d)$ 
(and so $g\in{\rm int}(\Theta_\mu)$).
Moreover, $h$ being non trivial, nonnegative and continuous, $h>0$ on some open set $O$ and so 
$\mu(O)>0$.
Therefore, by Lemma \ref{lemma-aux2}(b), $f_h$ is twice differentiable and strictly convex on ${\rm int}(\P[\x]_d)$
and (\ref{aux2-1})-(\ref{aux2-2}) yield (\ref{coro1-1}) and (\ref{coro1-3}), while (\ref{coro1-2}) follows from
(\ref{coro1-1}) and Theorem \ref{thmain}. That $\nabla^2f_h(g)$ is the matrix of
$2d$-moments of the measure $d\nu=h\exp(-g)d\x$ follows from (\ref{coro1-3}). This proves (b).

To prove (a) when $h$ is not nonnegative, write $h=h^+-h^-$ with $h^+:=\max[0,h]$ and
$h^-:=\max[0,-h]$. Both $h^+$ and $h^-$ are continuous PHFs of degree $p$, and nonnegative. 
Moreover, $f_h=f_{h^+}-f_{h^-}$ and 
so applying  (b) to $f_{h^+}$ and $f_{h^-}$, yields (a) by additivity.

\end{proof}

\begin{remark}
\label{remark-PHF}
(a) Notice that proving convexity of $f_h$ directly from its definition (\ref{coro1-0}) is not obvious at all whereas it becomes 
much easier when using Theorem \ref{thmain}.

(b) In Lemma \ref{lemma-aux2}, differentiability of $f$ on the convex cone $C_d$  should be now
in the sense of G\^ateaux-differentiablity, not explored here. 
\end{remark}

We end up with the following relatively surprising results which even though are again particular cases of
Lemma \ref{lemma0}, deserve special mention.

\begin{lemma}
\label{lemma4}
Let $y\geq0$ be fixed, let $h$ be a PHF of degree $p\in\R$ and let  $\xi,\psi:\R_+\to\R$ be 
measurable functions such that
\[\dis\int_{\{t:\psi(t^d)\leq y\}}t^{n+p-1}\xi(t^d)\,dt\,<\,+\infty.\]
Let $f_h:C_d\to\R$, $0\neq d\in\R$, be the function:
\[g\mapsto f_h(g)\,:=\,\int_{\{\x\,:\,\psi(g(\x))\leq y\}}\xi(g(\x))\,h(\x)\,d\x,\qquad g\in C_d.\]
Then whenever $\int \vert h\vert \exp(-g)d\x<+\infty$, $f_h(g)$ is finite, and
\begin{equation}
\label{lemma4-1}
f_h(g)\,=\,\underbrace{\frac{d\dis\int_{\{t:\psi(t^d)\leq y\}}t^{n+p-1}\xi(t^d)\,dt}{\Gamma((n+p)/d)}}_{{\rm cte}(y,p,d,\xi,\psi)}\cdot
\int_{\R^n}h\exp(-g)d\x,
\end{equation}
where the constant depends only on $\xi,\psi,d,p,y$ and neither on $g$ nor $h$. Therefore,
$f_h$ is convex whenever $h$ is nonnegative (and strictly convex on $C^0_d$ whenever $h>0$ on $\R^n\setminus\{0\}$). In particular,

\begin{eqnarray}
\label{coroeuler-1}
\int_{\R^n} gh\,\exp(-g)\,d\x&=&\frac{n+p}{d}\int_{\R^n}h\,\exp(-g)\,d\x\\
\label{coroeuler-2}
\int_{\{\x\,:\,g(\x)\leq 1\}}  gh\,d\x&=&\frac{1}{\Gamma((n+p)/d)}\int_{\R^n}h\,\exp(-g)\,d\x\\
\label{lemma4-2}
\frac{\dis\int_{\{\x\,:\,g(\x)\leq y\}}  \exp(-g)\,d\x}{\dis\int_{\R^n}\,\exp(-g)\,d\x}&=&\frac{\dis\int_0^y\,\exp(-z)z^{n/d-1}\,dz}
{\Gamma(n/d)}\\
\label{lemma4-3}
\int_{\{\x\,:\,g(\x)\leq y\}}  \exp(g)\,d\x&=&\frac{\dis\int_{\R^n}\,\exp(-g)\,d\x}{\Gamma(n/d)}\dis\int_0^y\,\exp(z)z^{n/d-1}\,dz.
\end{eqnarray}
\end{lemma}

\begin{proof}
We first assume that $h$ is nonnegative and $\int h\exp(-g)<+\infty$, so that
(\ref{lemma4-1}) follows from Lemma \ref{lemma0} with $t\mapsto \phi(t):=\xi(t)I_{[0,y]}(\psi(t))$, which yields
$f_h(g)={\rm cte}(y,p,d,\xi,\psi)\cdot A(g,h)$, with $A(\cdot,\cdot)$ as in (\ref{A}) and
\[{\rm cte}(y,p,d,\xi,\psi)\,=\,\int_{\{t:\psi(t^d)\leq y\}}t^{n+p-1}\xi(t^d)\,dt,\]
and the result follows by recalling that with $\phi(t)=\exp(-t)$ one had already obtained in (\ref{AA})
\[A(g,h)\,=\,\frac{d}{\Gamma((n+p)/d)}\,\int_{\R^n}h\,\exp(-g)\,d\x\:(\,<\,+\infty).\]
When $h$ is not nonnegative, writing $h=h^+-h^-$ where both $h^+$ and $h^-$ are also PHFs of degree $p$,
the result follows by additivity since $\int \vert h\vert \exp(-g)d\x<+\infty$ only if both
$\int h^+\exp(-g)d\x$ and $\int h^-\exp(-g)d\x$ are finite.

Finally, (\ref{coroeuler-1})-(\ref{lemma4-3}) are special cases of (\ref{lemma4-1}) with
respective choices $t\mapsto\psi(t):=0,\xi(t):=\exp(-t)$, then
$t\mapsto \psi(t):=t,\xi(t)=t$ and finally, $t\mapsto \psi(t)=t,\xi(t):=\exp(-t)$ and $t\mapsto \xi(t):=\exp(t)$.

At last, when $h$ is nonnegative, convexity and strict convexity follow from Lemma 
\ref{lemma-aux2}(a) with $d\mu=hd\x$ (and so $\supp\mu=\R^n$ if $h>0$).
\end{proof}

So Lemma \ref{lemma4} shows that $f$ is convex provided that $h$ is nonnegative and no matter 
how the functions $\xi$ and $\psi$ behave! 

Next, if $\mu$ is the non Gaussian measure
$d\mu=\exp(-g)d\x$ on $\R^n$, then (\ref{lemma4-2}) shows how fast $\mu(\{\x:g(\x)\leq y\})$
converges to the non Gaussian integral $\int_{\R^n}\exp(-g)d\x$ as $y\to\infty$. It converges as fast
as the one-dimensional integral $\int_0^yt^{n/d-1}\exp(-t)dt$ converges to the Gamma function $\Gamma(n/d)$.

\subsection{Polarity}

We here investigate the {\it polar} $G^\circ$ set of the sublevel set $G:=\{\x\,:\,g(\x)\leq1\}$ assumed to be compact 
and when $g$ is a proper closed convex PHF. In this case $G$ is a convex body, and in fact $g$ is a {\it gauge}.

The polar $C^\circ$ of a set $C\subset\R^n$ is the convex set defined by
\[C^\circ\,=\,\{\x\in\R^n\::\: \sigma_C(\x)\,\leq\,1\}\quad\mbox{with }\sigma_C(\x):=\sup_\y\,\{\langle\x,\y\rangle\,:\,\y \in C\}.\]
and $(C^\circ)^\circ$ is the smallest convex balanced set that contains $C$. 

Recall that the Legendre-Fenchel {\it conjugate} $f^*:\R^n\to\R\cup\{+\infty,-\infty\}$ of $f:\R^n\to\R$ is defined by
\[f^*(\u)\,:=\,\sup_\x\:\{\langle\u,\x\rangle -f(\x)\,\}\qquad \u\in\R^n.\]
The conjugate $g^*$ of a PHF $g$ of degree $d\in\R$ is tself
a PHF of degree $q$ with $\frac{1}{d}+\frac{1}{q}=1$
(where $d$ does not need to be positive); see e.g. Lasserre \cite{jca}.

\begin{proposition}
\label{th-polar}
Let $g$ be closed proper convex PHF of degree $1<d\in\R$, and 
let $G:=\{\x\,:\,g(\x)\leq1/d\}$. Then with $\frac{1}{d}+\frac{1}{q}=1$,
\begin{eqnarray}
\label{th-polar-1}
G^\circ&=&\{\x\in\R^n\::\: g^*(\x)\,\leq\,1/q\},
\end{eqnarray}
where $g^*$ is the Legendre-Fenchel conjugate of $g$. In other words, $G^\circ$ is a sublevel set the PHF $g^*$ of degree $q$. Moreover, if $G$ is bounded then
\begin{equation}
\label{th-polar-3}
{\rm vol}\,(G^\circ)\,=\,\frac{1}{q^{n/q}\Gamma(1+n/q)}\int_{\R^n}\exp(-g^*)\,d\x.
\end{equation}
\end{proposition}
\begin{proof}
(\ref{th-polar-1}) is from Rockafellar \cite[Corollary 15.3.2]{rockafellar} and (\ref{th-polar-3}) 
follows from Theorem \ref{thmain} applied to the PHF $g^*$ and with $y=1/q$.
\end{proof}
\begin{example}
\label{ex1}
Let $x\mapsto g(x):=\vert x\vert^3$ if $x>0$ and $+\infty$ otherwise. Hence, 
$d=3$, $q=3/2$,
$G=[-1,1]$, $\sigma_G(x)=\vert x\vert$, and $g^*(x)=\frac{2}{3\sqrt{3}}\vert x\vert^{3/2}$.
One retrieves that $G^\circ=G=[-1,1]=\{x:g^*(x)\leq (d-1)/d^q\}$.
\end{example}
\begin{example}
\label{ex2}
Let $x\mapsto g(x):=\vert x\vert$ so that $G=[-1,1]$. 
As $g^*(x)=0$ if $x\in [-1,1]$ and $+\infty$ otherwise
(a PHF of degree $0$) one may check that indeed $G^\circ=[-1,1]=G$ and
(\ref{th-polar-1}) holds although $d=1$ and $g$ is not strictly convex.
\end{example}
\begin{example}
\label{ex3}
Let $\x\mapsto g(\x):=x_1^4+x_2^4$ and $G=\{\x:\,x_1^4+x_2^4\leq 1/4\}$.
Then $g^*(\x)=3 (x_1^{4/3}+x_2^{4/3})/4^{4/3}$
(a PHF of degree $4/3$) and
$G^\circ=\{\x:\,x_1^{4/3}+x_2^{4/3}\leq 1/4^{1/3}\}$.
\end{example}

\subsection{A variational property of homogeneous polynomials}

We end up this section with an intriguing variational property of homogeneous polynomials that are sums of squares.

Let $\tv_d(\x)$ be the vector of all monomials $(\x^\alpha)$ of degree $d$ and
let $g\in\R[\x]_{2d}$ be homogeneous and a sum of squares, that is,
\begin{equation}
\label{newdef}
g(\x)\,=\,-\frac{1}{2}\tv_d(\x)^T\si\tv_d(\x),\qquad \x\in\R^n,
\end{equation}
for some real symmetric $\ell(d)\times \ell(d)$ matrix $\si$ which is positive definite (denoted $\si\succ0$). If $d=1$  it is well-known that 
\[\int_{\R^n} \exp(-g)d\x\,=\,\frac{(2\pi)^{n/2}}{\sqrt{\det\,\si}},\]
and
\[\int_{\R^n} \tv_d(\x)\tv_d^T(\x)\exp(-g)d\x\,=\,\frac{(2\pi)^{n/2}}{\sqrt{\det\,\si}}\si^{-1},\]
that is, $\si^{-1}$ is the covariance matrix associated with the Gaussian probability density $(2\pi)^{-n/2}\sqrt{\det \si}\exp(-g)$.

When $d>1$ the non Gaussian integral can be still expressed as a (possibly complicated) 
combination of several algebraic invariants of $g$, but in general not in terms of the single algebraic invariant $\det\si$.

It turns out that $\det\si$ and $\int \exp(-\frac{1}{2}\v_d(\x)^T\si\v_d(\x))d\x$ are still related in the following 
Gaussian-like manner.
Let $\s_{\ell(d)}^{++}\subset\s_{\ell(d)}$ be the convex cone of $\ell(d)\times \ell(d)$ positive definite matrices, and
let $\theta_d: \s_{\ell(d)}^{++}\to \R$ be defined by:
\begin{equation}
\label{def-f2}
\theta_d(\si)\,:=\,(\det\si)^k\dis\int_{\R^n} \exp(-k\tv_d(\x)^T\si\,\tv_d(\x))\,d\x,\qquad\si\in\s_{\ell(d)}^{++},\end{equation}
where $k=n/(2d\ell(d))$ and let
\[\M_d(\si)\,:=\,\frac{\dis\int_{\R^n}\tv_d(\x)\tv_d(\x)^T\exp(-k\tv_d(\x)^T\si\,\tv_d(\x))\,d\x}
{\dis\int_{\R^n}\exp(-k\tv_d(\x)^T\si\,\tv_d(\x))\,d\x},\]
be the matrix of moments of order $d$, associated with the non Gaussian probability measure
\begin{equation}
\label{defmu}
\mu(B)\,:=\,\frac{\dis\int_B\exp(-k\tv_d(\x)^T\si\tv_d(\x))\,d\x}
{\dis\int_{\R^n}\exp(-k\tv_d(\x)^T\si\tv_d(\x))\,d\x},\qquad B\in\mathcal{B}.\end{equation}
Observe that $\theta_d$ is nonnegative and positively homogeneous of degree $0$; therefore $\theta_d$ is constant in any fixed direction $\si$.
In particular, if $d=1$ then $k=1/2$, $\mu$ is a Gaussian probability measure,
$\M_1(\si)$ is the associated covariance matrix $\si^{-1}$ and $\theta_d(\si)$ is constant.
In fact, 

\begin{lemma}
\label{lemma1}
Let $\theta_d$ be the function defined in (\ref{def-f2}). Then
$\langle \M_d(\si),\si\rangle=\ell(d)$ for all $\si$ in the domain of $\theta_d$ and
$\nabla \theta_d(\si)=0$ if 
\begin{equation}
\label{lemma1-1}
\M_d(\si)\,=\,\si^{-1}.\end{equation}
\end{lemma}
\begin{proof}
Let $g(\x)=k\tv_d(\x)\si\tv_d(\x)$ so that
\begin{eqnarray*}
\langle \M_d(\si),\si\rangle\int_{\R^n} \exp(-g)\,d\x&=&\left\langle \int_{\R^n} \tv_d(\x)\tv_d(\x)^T\exp(-g)\,d\x,\si\right\rangle\\
&=&k^{-1}\int_{\R^n}k \tv_d(\x)^T\si\tv_d(\x)\exp(-g)\,d\x\\
&=&k^{-1}\int_{\R^n}g\,\exp(-g)\,d\x\\
&=&\frac{k^{-1}n}{2d}\int_{\R^n}\exp(-g)\,d\x
\quad\mbox{[by (\ref{coroeuler-1})]},
\end{eqnarray*}
which yields the desired result $\langle \M_d(\si),\si\rangle=\ell(d)$. Next,
write the gradient $\nabla\theta(\si)$ in the form $A_1+A_2$ with
\begin{eqnarray*}
A_1&=&\nabla((\det\si)^k)\int_{\R^n}\exp(-k\tv_d(\x)^T\si\tv_d(\x))\,d\x\\
&=&k(\det\si)^{k-1}\si^\A\int_{\R^n}\exp(-k\tv_d(\x)^T\si\tv_d(\x))\,d\x\\
&=&k\frac{\si^\A}{\det\si}(\det\si)^k\int_{\R^n}\exp(-k\tv_d(\x)^T\si\tv_d(\x))\,d\x\\
&=&k\si^{-1}\,\theta_d(\si),
\end{eqnarray*}
where $\si^\A$ is the {\it adjugate} of $\si$ (see e.g. \cite[p. 411]{bernstein}), and
\begin{eqnarray*}
A_2&=&(\det\si)^k\:\nabla\left(\int_{\R^n}\exp(-k\tv_d(\x)^T\si\tv_d(\x))\,d\x\right)\\
&=&-k(\det\si)^k\int_{\R^n}\tv_d(\x)\tv_d(\x)^T\exp(-k\tv_d(\x)^T\si\tv_d(\x))\,d\x\\
&=&-k\M_d(\si)\,\theta_d(\si).
\end{eqnarray*}
This yields $A_1+A_2=k\theta_d(\si)(\si^{-1}-\M_d(\si))$ and so $\nabla \theta_d(\si)=0$ if $\M_d(\si)=\si^{-1}$. 
\end{proof}
Lemma \ref{lemma1} states that for all critical points $\si$, or equivalently
for all critical SOS homogeneous polynomials $g$ of the function $\theta_d$ (assuming 
that at least one such critical point exists), 
their associated non Gaussian measure $d\mu =\exp(-g)d\mu$ (rescaled to a probability measure)
has the Gaussian-like property
that $\si^{-1}$ is the ``$d$-covariance" matrix of $\mu$!

\section{sublevel set of minimum volume containing a compact set}

If $\K\subset\R^n$ is a convex body, computing the ellipsoid of minimum volume that contains $\K$
is a classical problem which has an optimal solution called the {\it L\"owner-John} ellipsoid; see e.g. Barvinok \cite[p. 209]{barvinok}. In this section we consider the following generalization:

{\em $\P$: Find a homogeneous polynomial $g$ of degree $2d$ such that its sublevel set
$G:=\{\x\,:\,g(\x)\leq 1\}$ contains $\K$ and has minimum volume among all such sublevel sets
with this inclusion property.}

With $\K\subset\R^n$, let $C_{2d}(\K)\subset\R[\x]_{2d}$ be the convex cone
of polynomials of degree at most $2d$ that are nonnegative on $\K$. Recall that $\P[\x]_{2d}\subset\R[\x]_{2d}$
is the convex cone of homogeneous polynomials of degree $2d$ with compact sublevel set $\{\x:g(\x)\leq 1\}$.
We next show that problem $\P$ is a convex optimization problem:

\begin{proposition}
\label{min-vol-lemma}
The minimum volume of a sublevel set
$\{\x:g(\x)\leq1\}$, $g\in\P[\x]_{2d}$, that contains $\K\subset\R^n$ is $\rho/\Gamma(1+n/2d)$ where $\rho$ is the optimal value of the finite-dimensional convex optimization problem:
\begin{equation}
\label{minvolume}
\mathcal{P}:\qquad \rho=\dis\inf_{g\in\P[\x]_{2d}}\:\left\{\,\int_{\R^n} \exp(-g)\,d\x\::\: 1-g\,\in\,C_{2d}(\K)\,\right\}.
\end{equation}
\end{proposition}
\begin{proof}
From Theorem \ref{thmain}
\[\vol(\{\x\,:\,g(\x)\leq 1\})\,=\,\frac{1}{\Gamma(1+n/2d)}\int_{\R^n}\exp(-g)\,d\x.\]
Moreover, the sublevel set $\{\x\,:\,g(\x)\leq 1\}$ contains $\K$ if and only if
$1-g\in C_{2d}(\K)$, and so $\rho/\Gamma(1+n/2d)$ in (\ref{minvolume}) is the minimum value of 
all volumes of sublevels sets $\{\x\,:\,g(\x)\leq 1\}$, $g\in\P[\x]_{2d}$, that contain $\K$.
Now since $g\mapsto \int_{\R^n}\exp(-g)d\x$ is strictly convex (see Corollary \ref{coro-convex}(b)) and $C_{2d}(\K)$ is a convex cone,
problem $\mathcal{P}$ is a finite-dimensional convex optimization problem.
\end{proof}
We also have the following characterization of an optimal solution of $\mathcal{P}$ when it exists.
Let $M(\K)$ be the convex cone of finite Borel measures on $\K$.
\begin{thm}
\label{vol-suff-cond}
Let $\K\subset\R^n$ be compact and consider the convex optimization problem $\mathcal{P}$ in (\ref{minvolume}).

{\rm (a)} Suppose that
$g^*\in\P[\x]_{2d}$ is an optimal solution of $\mathcal{P}$. Then there exists $\mu^*\in M(\K)$ 
such that
\begin{equation}
\label{kkt-suff}
\int_{\R^n}\x^\alpha\exp(-g^*)d\x\,=\,\int_{\K}\x^\alpha\,d\mu^*,\quad\forall\vert\alpha\vert=2d;\quad\int_\K(1-g^*)\,d\mu^*=0.\end{equation}
In particular, $\mu^*$ is supported on the real variety $V:=\{\x\in\K: g(\x)=1\}$ and in fact,
$\mu^*$ can be substituted with another 
measure $\nu^*\in M(\K)$ supported on at most ${n+2d-1\choose 2d}+1$ points of $V$.

{\rm (b)} Conversely, if $g^*\in\P[\x]_{2d}$ and $\mu^*\in M(\K)$ satisfy (\ref{kkt-suff}) then
$g^*$ is an optimal solution of $\mathcal{P}$.
\end{thm}
\begin{proof}
(a) We may and will consider $g^*$ as an element of $\R[\x]_{2d}$ with 
$g^*_\beta=0$ whenever $\vert\beta\vert<2d$.
As $\K$ is compact, there exists $\theta\in\P[\x]_{2d}$ such that
$1-\theta\in{\rm int}\,C_{2d}(\K)$, i.e., Slater's condition holds for $\mathcal{P}$. 
Indeed, choose $\theta:=M^{-1}\Vert\x\Vert^{2d}$ for $M>0$ sufficiently large so
that $1-\theta>0$ on $\K$. Hence with $\Vert g\Vert_1$ 
denoting the $\ell_1$-norm of the coefficient vector of $g$ (in $\R[\x]_{2d}$),
there exists $\epsilon>0$ such that
for every $h\in B(\theta,\epsilon)(:=\{h\in\R[\x]_{2d}:\Vert \theta-h\Vert_1<\epsilon\}$), 
$1-h>0$ on $\K$. 

Therefore, the optimal solution $g^*$ satisfies the KKT-optimality conditions, which read:
\begin{equation}
\label{aux60}
\int_{\R^n}\x^\alpha\,\exp(-g^*)\,d\x=y^*_\alpha,\quad\forall\vert\alpha\vert=2d;\quad 
\langle 1-g^*,\y^*\rangle =0,\end{equation}
for some $\y^*=(y^*_\alpha)$, $\alpha\in\N^n_{2d}$, an element of the dual cone $C_{2d}(\K)^*\subset\R^{s(2d)}$ of $C_{2d}(\K)$. By Lemma \ref{dual-ck} in \S \ref{dualck},
\[C_{2d}(\K)^*\,=\,\{\y\in\R^{s(2d)}\::\: \exists\mu^*\in M(\K)\mbox{ s.t. }y_\alpha=\int_\K\x^\alpha\,d\mu^*,\:\alpha\in\N^n_{2d}\:\},\]
and so (\ref{kkt-suff}) is just (\ref{aux60}) restated in terms of $\mu^*$.
Finally, the last statement follows from Anastassiou
\cite[Theorem 2.1.1, p. 39]{anastassiou} applied to the ${n+2d-1\choose 2d}$ equality constraints of (\ref{kkt-suff}).

(b) As Slater's condition holds for $\mathcal{P}$, the KKT-optimality conditions (\ref{kkt-suff}) are sufficient
to ensure that $g^*$ is an optimal solution on $\mathcal{P}$.
\end{proof}
Theorem \ref{vol-suff-cond} states that when (\ref{kkt-suff}) holds, there is an optimal solution $g^*$ to $\mathcal{P}$
such that $g(\x)=1$ on (at most) ${n+2d-1\choose 2d}+1$ points of $\K$, the analogue for $d>1$ of
the well-known property of the L\"owner-John ellipsoid in the case $d=1$.\\

Even though being convex and finite-dimensional, $\mathcal{P}$ is by no means easy to solve because there is no simple and computationally tractable way of describing the convex cone $C_{2d}(\K)$. However, there are cases where one may provide a sequence of inner or outer approximations that both converge to $C_{2d}(\K)$. One such case is when one knows 
all moments of a finite Borel measure whose support is exactly $\K$, and another case is when 
$\K=\{\x\,:\,g_k(\x)\geq 0,\:k=1,\ldots,m\}$ for some polynomials $(g_k)\subset\R[\x]$, i.e.,
$\K$ is a compact basic semi-algebraic set.

\subsection{Lower bounds via inner approximations}
Suppose that one knows all moments $\z=(z_\alpha)$, $\alpha\in\N^n$, of a finite Borel measure $\mu$ 
on $\K$, i.e.,
\[z_\alpha\,=\,\int_\K \x^\alpha\,d\mu(\x),\qquad\forall\alpha\in\N^n,\]
whose support is exactly $\K$.
For every $k\in\N$ and $p\in\R[\x]$, let $\M_k(p,\,\z)$ be the localizing matrix with respect to the polynomial $p$ and the moment sequence $\z$, that is, 
$\M_k(p,\,\z)$ is the $s(k)\times s(k)$ real symmetric matrix with rows and columns indexed in the
canonical basis $(\x^\alpha)$, $\alpha\in\N^n_k$, of $\R[\x]_k$, and with entries
\begin{eqnarray*}
\M_k(p,\,\z)(\alpha,\beta)&=&\int_\K p(\x)\,\x^{\alpha+\beta}\,d\mu(\x),\qquad\forall\alpha,\beta\in\N^n_k\\
&=&\sum_\gamma p_\gamma\,z_{\alpha+\beta+\gamma}
\end{eqnarray*}
(when $p(\x)=\sum_\gamma p_\gamma\x^\gamma$).
We recall the following result.
\begin{lemma}(\cite[Theorem 3.2]{lasserresiopt})
\label{newlooklemma}
Let $\K\subset\R^n$ be compact and let $\mu$ be a finite Borel measure with
support $\K$ and with moments $\z=(z_\alpha)$, $\alpha\in\N^n$. Then $p\in\R[\x]$ is nonnegative on $\K$ if and only if $\M_k(p,\,\z)\succeq0$ for every $k=0,1,\ldots$.
\end{lemma}
In view of Lemma \ref{newlooklemma}, a natural idea is to relax the ``difficult" constraint
$1-g\in C_{2d}(\K)$ in (\ref{minvolume}) to $\M_k(1-g,\,\z)\succeq0$ for fixed $k$, and then let $k\to\infty$. Indeed, for every fixed $k$, the latter is much easier to handle
as it defines a spectrahedron\footnote{A spectrahedron is a convex set that can be formed by intersecting 
the cone of positive semidefinite matrices $\s_n$ with a linear affine subspace.} $\Delta_k\subset\R^{\ell(2d)}$ on the coefficients of the homogeneous polynomial $g$. 
And so one obtains a hierarchy of convex relaxations of $\mathcal{P}$ by minimizing the (strictly) convex function 
$g\mapsto \int\exp(-g)$ on the spectrahedra $\Delta_k$, $k\in\N$,
which yields a monotone nondecreasing sequence of {\it lower bounds} $\rho_k\leq\rho$, $k\in\N$, on $\rho$. Of course the larger $k$ (i.e., the larger the size of the localizing matrix $\M_k(1-g,\,\z)$) the better the lower bound 
$\rho_k$ (but also the harder the problem).

\begin{thm}
\label{th-inner}
Let $\K\subset\R^n$ be compact with nonempty interior and consider the finite-dimensional convex optimization problem:
\begin{equation}
\label{minvolume-d}
\mathcal{P}_k:\qquad \rho_k=\dis\inf_{g\in\P[\x]_{2d}}\:\left\{\,\int_{\R^n} \exp(-g)\,d\x\::\: \M_k(1-g,\,\z)\succeq0\,\right\}.
\end{equation}
The sequence $(\rho_k)$, $k\in\N$, is monotone nondecreasing with $\rho_k\leq\rho$ and:

{\rm (a)} If $\mathcal{P}_k$ has an optimal solution $g^*_k\in\P[\x]_{2d}$ then
there exists a SOS polynomial $\sigma^*_k\in\Sigma[\x]_k$ such that
\begin{eqnarray}
\label{dualpb1}
\int_{\R^n}\x^\alpha\,\exp(-g^*_k)\,d\x&=&\int_\K\x^\alpha\,\sigma^*_k\,d\mu,\qquad \forall \vert\alpha\vert=2d.\\
\label{dualpb2}
\int_\K(1-g^*_k)\,\sigma^*_k\,d\mu&=&0,
\end{eqnarray}
and $\rho_k=\frac{2d}{n}\dis\int_\K\sigma^*_kd\mu$.\\

{\rm (b)} Conversely if (\ref{dualpb1})-(\ref{dualpb2}) holds for some $g^*_k\in\P[\x]_{2d}$
and some $\sigma^*_k\in\Sigma[\x]_k$ then $g^*_k$ is an optimal solution of $\mathcal{P}_k$.\\

{\rm (c)} If in addition,
\[\sup_k\,\sup_{\vert\alpha\vert=2d}\:\vert g^*_{k\alpha}\,\vert\,<\,M,\]
for some $M>0$, then $\rho_k\to\rho$ and $\mathcal{P}$ has an optimal solution $g^*$.
\end{thm} 
For a proof see \S \ref{proof-inner}. So problem $\mathcal{P}_k$ amounts to minimize a strictly convex function on a spectrahedron
of $\R^{\ell(2d)}$.
For instance, one may use {\it interior point} methods  and minimize the standard log-barrier function
\[g\mapsto \phi_\nu(g)\,:=\,\int_{\R^n}\exp(-g)d\x-\frac{1}{\nu}\log{\rm det}(\M_k(1-g,\,\z)),\]
with parameter $\nu$, and let $\nu\to\infty$. For more details on log-barrier methods, 
the reader is referred to e.g. Wright \cite{wright}.

\subsection{Upper bounds via outer approximations}

Let $\K\subset\R^n$ be the compact basic semi-algebraic set defined by 
\[\K\,=\,\{\x\::\: u_j(\x)\geq0,\:\quad j=1,\ldots,m\},\]
for some polynomials $(u_j)\subset\R[\x]$. As $\K$ is compact, with no loss of generality we may and
will suppose that $u_1(\x)=M-\Vert\x\Vert^2$ for some $M$ large enough.
Let
\[Q^*\,=\,\{ \sum_{j=0}^m\sigma_j u_j\::\: u_j\in\Sigma[\x],\:j=0,\ldots,m\},\]
be the quadratic module of $\R[\x]$ generated by the $u_j$'s (with $u_0=1$).
$Q^*$ is Archimedean because the quadratic polynomial
$\x\mapsto M-\Vert\x\Vert^2$ belongs to $Q^*$. If one defines
\[Q^*_k\,:=\,\{ \sum_{j=0}^m\sigma_j u_j\::\: u_j\in\Sigma[\x];\quad {\rm deg}\,\sigma_j u_j \leq 2k,\:j=0,\ldots,m\},\quad k\in\N,\]
(a subset of $Q^*$ with a degree bound on the SOS weights $\sigma_j$) 
then this time one may replace the ``difficult" constraint $1-g\in C_{2d}(\K)$ by the stronger constraint
$1-g\in Q^*_k$ for fixed $k\in\N$.

The convex cone $Q^*_k$ is the dual cone of the closed pointed convex cone
\begin{equation}
\label{cone-Q_k}
Q_k:=\{\y\in\R^{s(2k)}\::\:\M_k(\y)\,\succeq0,\quad \M_{k-v_j}(u_j,\,\y)\succeq0,\:j=1,\ldots,m\},\end{equation}
where $\M_k(\y)$ (resp. $\M_{k-v_j}(u_j,\,\y)$) is the moment matrix
associated with the sequence $\y$ (resp. the localizing matrix associated with
$\y$ and the polynomial $u_j$). Hence, $Q^*_k$ has a nonempty interior; see e.g. Rockafellar \cite{rockafellar}.

Then solving the problem
\begin{equation}
\label{minvolume-outer}
\mathcal{P}'_k:\quad\dis\inf_{g\in\P[\x]_{2d}}\:\left\{\int_{\R^n}\exp(-g)d\x\::\: 1-g\in Q^*_k\right\},\quad k\in\N,\end{equation}
$k\geq d$, now provides a monotone sequence of  {\it upper bounds} $\rho'_k\geq\rho$, $k\in\N$.

\begin{thm}
\label{th-outer}
Let $\K\subset\R^n$ be compact with nonempty interior and consider the finite-dimensional convex optimization problem
(\ref{minvolume-outer}), $k\geq d$. 

{\rm (a)} The sequence $(\rho'_k)$, $d\leq k\in\N$, is monotone nonincreasing with $\rho'_k\to \rho$ as $k\to\infty$.
Moreover there exists an optimal solution $g^*_k\in\P[\x]_{2d}$.

In addition, assume that there exists $g_0\in \P[\x]_{2d}$ such that
$1-g_0\in {\rm int}\,Q^*_{k_0}$ for some $k_0\geq d$. Then:

{\rm (b)} If $k\geq k_0$ and $g^*_k\in\P[\x]_{2d}$ is an optimal solution of $\mathcal{P}'_k$,
 there exists a
vector $\y^k\in Q_k$ such that:
\begin{equation}
\label{th-outer-1}
0\,=\,\langle 1-g^*_k,\y^k\rangle;\quad y^k_\alpha\,=\,\int_{\R^n}\x^\alpha\,\exp(-g^*_k)\,d\x,\quad\forall\,\vert\alpha\vert=2d.
\end{equation}
\indent
{\rm (c)} Conversely, if $k\geq k_0$ and $(\y^k,1-g^*_k)\in Q_k\times Q^*_k$ satisfy (\ref{th-outer-1}), then
$g^*_k$ is an optimal solution of $\mathcal{P}'_k$.

{\rm (d)} Let $(g^*_k)\subset\P[\x]_{2d}$, $k\geq k_0$, be a sequence of optimal solutions.
If
\[\sup_k\,\sup_{\vert\alpha\vert=2d}\:\vert g^*_{k\alpha}\,\vert\,\,<\,N,\]
for some $N>0$, then every accumulation point $g^*$ of the sequence 
$(g^*_k)$, $k\in\N$, is an optimal solution of  $\mathcal{P}$.
\end{thm} 
For proof see \S \ref{proof-outer}. 

\section{Proofs}

\subsection{Proof of Theorem \ref{th-inner}}
\label{proof-inner}
\begin{proof}
(a) That the sequence $(\rho_k)$ is monotone non decreasing is straightforward since the constraints of $\mathcal{P}_k$
are more and more restrictive as $k$ increases. Next, as $\K$ is compact there exists $M,\delta>0$ such that $M-\Vert\x\Vert^{2d}\,>\delta$ for all $\x\in\K$.
With $g_0\in\Sigma[\x]_d$ being the polynomial $\x\mapsto M^{-1}\Vert \x\Vert^{2d}$, one has 
$1-g_0(\x)>\delta/M$ for all $\x\in\K$ and so $\M_k(1-g_0,\,\z)\succ0$. Indeed, for every $0\neq f\in\R[\x]_k$,
\[\langle \f,\M_k(1-g_0,\,\z)\,\f\rangle\,=\,\int_\K f^2(1-g_0)\,d\mu\,\geq\,\frac{\delta}{M}\,\int_\K f^2d\mu\,>\,0.\]
where the last inequality is because $\K$ has nonempty interior and ${\rm supp}\,\mu=\K$.
Observe also that $g_0\in\P[\x]_{2d}$ and so $g_0$ is a strictly feasible solution
of $\mathcal{P}_k$, that is, Slater's condition\footnote{For a convex optimization
$\inf_\x\{f(\x)\,:\,h_k(\x)\geq0,\:k=1,\ldots,m\}$ Slater's condition holds if there exists $\x_0$ such that
$h_k(\x_0)>0$ for all $k$. In this case the Karush-Kuhn-Tucker (KKT) optimality
conditions are necessary and sufficient.} holds for $\mathcal{P}_k$. So 
the Karush-Kuhn-Tucker (KKT) optimality conditions at a point
$g^*_k$ are necessary and sufficient for $g^*_k$ to be a (global) minimizer of $\mathcal{P}_k$. Therefore, there exists $0\preceq\Delta\in\s_{s(k)}$ such that:
\begin{equation}
\label{aaux1}
\int_{\R^n}-\x^\alpha\exp(-g^*_k)\,d\x+\langle \Delta,\M_k(\x^\alpha,\,\z)\rangle\,=\,0,\qquad \forall \vert\alpha\vert=2d,
\end{equation}
and
\begin{equation}
\label{aux2}
\langle \M_d(1-g^*_k,\,\z),\Delta\rangle\,=\,0,\end{equation}
where one has written $\M_k(g^*_k\,\z)=\sum_\alpha g^*_{k\alpha}\,\M_k(\x^\alpha,\,\z)$ (with $\M_k(\x^\alpha,\,\z)$
being the localizing matrix with respect to the polynomial $\x^\alpha$ and the moment sequence $\z$).
From its spectral decomposition, $\Delta=\sum_jq_jq_j^T$ for some vectors  $(q_j)$, which yields
\[\langle\Delta,\M_k(\x^\alpha,\,\z)\rangle\,=\,\int_\K\sigma^*_k(\x)\,\x^\alpha\,d\mu,\qquad\forall\vert\alpha\vert=2d,\]
where $\sigma^*_k=\sum_j(q_j^T\v_k(\x))^2\in\Sigma_k$ is a SOS polynomial.
And so (\ref{aaux1}) yields (\ref{dualpb1}),
and the complementary condition (\ref{aux2}) yields (\ref{dualpb2}).
Next, multiplying both sides of (\ref{aaux1}) with $g^*_{k\alpha}$ and summing up yields:
\begin{eqnarray*}
\int_{\R^n}g^*_k\,\exp(-g^*_k)\,d\x\:\left(=\,\frac{n}{2d}\int_{\R^n} \exp(-g^*_k)\,d\x\right)&=&\langle \Delta,\M_k(g^*_k,\,\z)\rangle\\
&=&\int_\K\sigma^*_k\,g^*_k\,d\mu,\\
&=&\int_\K\sigma^*_k\,d\mu \quad\mbox{[by (\ref{aux2})]}
\end{eqnarray*}

(b) The converse is because under Slater's condition, the KKT optimality conditions are also sufficient for $g^*_k$ to be a minimizer.

(c) Write $g^*_k(\x)=\sum_{\vert\alpha\vert=2d}g^*_{k\alpha}\x^\alpha$.
As $\sup_k\sup_{\vert\alpha\vert=2d}\vert g^*_{k\alpha}\vert<M$, there exists a subsequence $(k_i)$, $i\in\N$, 
and a vector $g^*\in\R^{\ell(2d)}$ such that
for every $\alpha\in\N^n$ with $\vert\alpha\vert=2d$, $g^*_{k_i\alpha}\to g^*_\alpha$ as $i\to\infty$.

Notice that $\x\mapsto g^*(\x)=\limsup_{i\to\infty}g^*_{k_i}(\x)=\liminf_{i\to\infty}g^*_{k_i}(\x)$.
Next, as $\rho\geq \rho_k$ for every $k$,  and
$\exp(-g^*_k)\geq0$, by Fatou's Lemma (see e.g. Ash \cite{ash})
\begin{eqnarray*}
\rho\,\geq\,\liminf_{i\to\infty}\rho_{k_i}&=&\liminf_{i\to\infty}\int_{\R^n} \exp(-g^*_{k_i})\,d\x\\
&\geq&\int_{\R^n}\liminf_{i\to\infty} \exp(-g^*_{k_i})\,d\x\\
&=&\int_{\R^n}\exp(-\limsup_{i} g^*_{k_i})\,d\x=
\int_{\R^n}\exp(-g^*)\,d\x.
\end{eqnarray*}
On the other hand, observe that for every $k\in\N$, $\M_k(1-g^*_k,\,\z)\succeq0$ implies
$\M_j(1-g^*_k,\,\z)\succeq0$ for all $j\leq k$. Hence, let $j$ be fixed arbitrary so that
$\M_j(1-g^*_k,\,\z)\succeq0$, for all sufficiently large $k$.
For every $f\in\R[\x]_{j}$,  $f^2g_k$ is uniformly bounded on $\K$
and so by Fatou's Lemma
\begin{eqnarray*}
0\leq\limsup_{i\to\infty}\int_{\K} f^2\,(1-g^*_{k_i})\,d\mu&=&-\liminf_{i\to\infty}\int_{\K}f^2\,(g^*_{k_i}-1)\,d\mu\\
&\leq&\int_\K -\liminf_{i\to\infty}f^2\,(g^*_{k_i}-1)\,d\mu\\
&=&\int_{\K}f^2\,(1-g^*)\,d\mu
.\end{eqnarray*}
As $f\in\R[\x]_j$ was arbitrary, $\M_j(1-g^*,\,\z)\succeq0$; and as $j$ was arbitrary,
$\M_j(1-g^*,\,\z)\succeq0$, for all $j=0,1,\ldots$ But by Lemma \ref{newlooklemma},
this implies $1-g^*\geq0$ on $\K$, and so $g^*$ is a feasible solution for $\mathcal{P}$. 
Combining with $\rho\geq \int_{\R^n}\exp(-g^*)d\x$ yields that $g^*$ is an optimal solution of $\mathcal{P}$.
\end{proof}

\subsection{Proof of Theorem \ref{th-outer}}
\label{proof-outer}

\begin{proof}
(a) The sequence $(\rho'_k)$ is monotone non increasing because $Q^*_k\subset Q^*_{k+1}$ for every $k\in\N$.
Next, let $\epsilon>0$ be fixed, and let $g\in \P[\x]_{2d}$ be such that $\rho\leq \int_{\R^n}\exp(-g)d\x\leq \rho+\epsilon$.
Observe that $\tilde{g}:=(1+\epsilon)^{-1}g\in\P[\x]_{2d}$ and
\begin{eqnarray*}
\int_{\R^n}\exp(-\tilde{g})\,d\x&=&\int_{\R^n}\exp(-(1+\epsilon)^{-1}g)\,d\x\\
&=&(1+\epsilon)^{n/d}\int_{\R^n}\exp(-g)\,d\x\leq (1+\epsilon)^{n/d}(\rho+\epsilon).\end{eqnarray*}
 and $1-\tilde{g}=\frac{1+\epsilon -g}{1+\epsilon}>0$ on $\K$.
 Therefore, as $Q^*$ is Archimedean, by Putinar's Positivstellensatz \cite{putinar},
 $1-\tilde{g}\in Q^*$, that is, $1-\tilde{g}\in Q^*_{k_1}$  for some $k_1$. Hence
 $\tilde{g}$ is a feasible solution of $\mathcal{P}_{k_1}$ which implies
 $\rho\leq \rho'_{k_1}\leq (1+\epsilon)^{n/d}(\rho+\epsilon)$.
 As the sequence is monotone and $\epsilon>0$ was arbitrary, we obtain the desired convergence 
 $\rho'_k\to\rho$.
 
 Next, let $\w=(w_\alpha)$, $\alpha\in\N^n$, be the sequence of moment of the Lebesgue measure
 on $\K$, and let $g\in\R[\x]_{2d}$ be an arbitrary  feasible solution.
 Let $\M_k(\w)$ (resp. $\M_{k-v_j}(u_j\,\w)$) be the moment (resp. localizing) matrix associated with $\w$ and $u_j$.
 Hence, $\M_d(\w)\succ0$ and $\M_{d-v_j}(u_j\,\w)\succ0$, $j=1,\ldots,m$, because $\K$ has nonempty interior.
 The constraint $1-g\in Q_k(u)$ which reads $(1-g)=\sigma_0+\sum_{j=1}^m\sigma_ju_j$ for some SOS polynomials $(\sigma_j)$
 has the equivalent form
 \begin{equation}
 \label{aux33}
 (1-g)_\alpha\,=\,\langle \X_0,\B_\alpha\rangle+\sum_{j=1}^m\langle\X_j,\C_{j\alpha}\rangle,\quad\forall\alpha\in\N^n_{2k},\end{equation}
 for some appropriate real symmetric $s(2(k-v_j))\times s(2(k-v_j))$ matrices $\X_j\succeq0$,
(and where $\B_\alpha,\C_{j\alpha}$ are given real symmetric matrices).
The vector of coefficients of the SOS polynomial $\sigma_j\in\Sigma[\x]$ is obtained from
the eigenvectors of $\X_j$, $j=0,\ldots,m$. For more details see e.g. \cite{lasserrebook2}.
Multiplying (\ref{aux33}) by $w_\alpha$ and summing up yield:
\begin{eqnarray*}
\int_\K(1-g)\,d\x&=&\langle \X_0,\sum_\alpha\B_\alpha w_\alpha\rangle
+\sum_{j=1}^m\langle \X_j,\sum_\alpha\C_{j\alpha}w_\alpha\rangle\\
&=&\langle \X_0,\M_k(\w)\rangle+\sum_{j=1}^m\langle \X_j,\M_{k-v_j}(u_j\,\w)\rangle\end{eqnarray*}
Observe that every feasible solution $g\in\R[\x]_{2d}$ is nonnegative otherwise $\int\exp(-g)dx$ is not bounded.
And so $\exp(-g)\geq 1-g$ on $\K$ because $g\geq0$. Therefore, for any minimizing sequence 
$(g_n,\X_j^n)$ of (\ref{minvolume-outer}),
\[\int_\K(1-g_n)\,d\x\,\leq\,\int_\K\exp(-g_n)\,d\x\,\leq \,\int_{\R^n}\exp(-g_n)\,d\x\,\leq\,\int_{\R^n}\exp(-g_0)\,d\x.\]
And so
\[\sup_n\left\{\langle \X^n_0,\M_k(\w)\rangle+\sum_{j=1}^m\langle \X^n_j,\M_{k-v_j}(u_j\,\w)\rangle\right\}\,\leq\,\int_{\R^n}\exp(-g_0)\,d\x.\]
As $\M_k(\w),\M_{k-v_j}(u_j\,\w)\succ0$ and $\X^n_j\succeq0$, $j=1,\ldots,m$, 
all matrices $\X^n_j$ are uniformly bounded. Hence one may extract a subsequence $(n_\ell)$ such that
$\X_j^{n_\ell}\to \X_j^*\succeq0$, $j=0,\ldots,n$, as $\ell\to\infty$. And so for every $\alpha\in\N^n_{2k}$,
\[\langle \X^*_0,\B_\alpha\rangle+\sum_{j=1}^m\langle\X^*_j,\C_{j\alpha}\rangle\,=\,
\lim_{\ell\to\infty}(1-g_{n_\ell})_\alpha\,=:\,(1-g^*)_\alpha,\]
for some homogeneous polynomial $g^*\in\R[\x]_{2d}$ (as all coefficients $g^*_\alpha$ with $\vert\alpha\vert\neq 2d$ vanish).
In other words $(g_{n_\ell})_\alpha\to g^*_\alpha$ as $\ell\to\infty$, for all $\alpha\in\N^n_{2k}$.
As $(1-g_{n_\ell})\geq0$ on $\K$ one also has $(1-g^*)\geq0$ on $\K$. Finally, since we also have 
the pointwise convergence $g_{n_\ell}(\x)\to g^*(\x)$ for all $\x\in\R^n$, invoking Fatou's lemma yields
\begin{eqnarray*}
\rho'_k\,=\,\lim_{\ell\to\infty}\int_{\R^n}\exp(-g_{n_\ell})\,d\x&\geq&\int_{\R^n}\liminf_{\ell\to\infty}\exp(-g_{n_\ell})\,d\x\\
&=&\int_{\R^n}\exp(-g^*)\,d\x,\end{eqnarray*}
which proves that $g^*$ is an optimal solution of (\ref{minvolume-outer}).

(b) $1-g_0\in {\rm int}\,Q^*_{k}$ for all $k\geq k_0$ because $1-g_0\in {\rm int}\,Q^*_{k_0}$ and
$Q^*_k\supset Q^*_{k_0}$. Hence Slater's condition holds for $\mathcal{P}'_k$ whenever $k\geq k_0$.
Therefore, if $g^*_k$ is an optimal solution of $\mathcal{P}_k$, there exists $\y^k\in Q_k$ such that the KKT-optimality condition hold, which yields (\ref{th-outer-1}).

(c) Follows from the fact that under Slater's condition, 
the KKT-optimality conditions are sufficient for $g^*_k$ to be an optimal solution of $\mathcal{P}'_k$.

(d) If $\sup_k\,\sup_{\vert\alpha\vert=2d}\:\vert g^*_{k\alpha}\,\vert<N$, let $g^*$ be an accumulation point,
i.e., a limit point of some subsequence $(g^*_{k_i})$, $i\in\N$, i.e.,
such  that for every $\alpha$ with $\vert\alpha\vert=2d$, $g^*_{k_i\alpha}\to g^*_\alpha$ as $i\to\infty$.
Let $g^*\in\R[\x]_{2d}$ be the homogeneous polynomial with coefficients $g^*_\alpha$, $\vert\alpha\vert=2d$.
For every $\x\in\K$, $(1-g_{k_i}(\x))\geq0$ because $1-g^*_{k_i}\in Q^*_{k_i}$ for every $i$. Therefore, with $\x\in\K$ fixed, \[0\leq 1-\lim_{i\to\infty}g^*_{k_i}(\x)\,=\,1-g^*(\x),\]
and so, as $\x\in\K$ is arbitrary,  $1-g^*\in C_{2d}(\K)$, which implies that $1-g^*$ is a feasible solution of 
$\mathcal{P}$. Next, from (a) and by Fatou's Lemma,
\begin{eqnarray*}
\rho\,=\,\lim_{i\to\infty}\rho_{k_i}\,=\,\liminf_{i\to\infty}\rho_{k_i}&=&\liminf_{i\to\infty}\int_{\R^n} \exp(-g^*_{k_i})\,d\x\\
&\geq&\int_{\R^n}\liminf_{i\to\infty} \exp(-g^*_{k_i})\,d\x\\
&=&\int_{\R^n}\exp(-g^*)\,d\x,
\end{eqnarray*}
and so as $g^*$ is feasible for $\mathcal{P}$, it is an optimal solution of $\mathcal{P}$.
\end{proof}

\subsection{The dual cone of $C_d(\K)$}
\label{dualck}

Recall that $M(\K)$ is the space of finite Borel measures on $\K$. 
Let $C_\infty(\K)\subset\R[\x]$ be the space of polynomials nonnegative on $\K\subseteq\R^n$ and
let $\Delta_\infty\subset \R[\x]^*$
be the set:
\begin{equation}\label{dual-ck-0}
\Delta_\infty\,:=\,\left\{\left(\int_\K \x^\alpha\,d\phi\right), \:\alpha\in \N^n\::\:\phi\in M(\K)\:\right\}.
\end{equation}
It is known that when $\K\subseteq\R^n$ is closed then 
$\Delta_\infty^*=C_\infty(\K)$ and $\Delta_\infty=C_\infty(\K)^*$. See e.g. the proof in Laurent \cite[Prop. 4.5]{laurent}
for $\K=\R^n$, which also works for any closed set $\K\subset\R^n$. The following result is 
a truncated version when $\K$ is compact.

\begin{lemma}
\label{dual-ck}
Let $\K\subset\R^n$ be compact. For every $d\in\N$, the dual $C_d(\K)^*$ of $C_{d}(\K)$ is the set:
\begin{equation}\label{dual-ck-1}
\Delta_d\,:=\,\left\{\left(\int_\K \x^\alpha\,d\phi\right), \:\alpha\in \N^n_d\::\:\phi\in M(\K)\:\right\}.
\end{equation}
\end{lemma}
\begin{proof}
For every $\y=(y_\alpha)\in\Delta_d$ and $f\in C_d(\K)$ with coefficient vector $\f\in\R^{s(d)}$:
\begin{equation}
\label{dual-100}
\langle\y,\f\rangle\,=\,\sum_{\alpha\in\N^n_d}f_\alpha y_\alpha
\,=\,\sum_{\alpha\in\N^n_d}\int_\K f_\alpha \x^\alpha\,d\phi
\,=\,\int_\K f\,d\phi\,\geq\,0.\end{equation}
Since (\ref{dual-100}) holds for all $f\in C_d(\K)$ and $\y\in\Delta_d$, then necessarily $\Delta_d\subseteq C_d(\K)^*$ and similarly,
$C_d(\K)\subseteq\Delta_d^*$. Next,
\begin{eqnarray*}
\Delta_d^*&=&\left\{\f\in\R^{s(d)}\::\:\langle\f,\y\rangle\geq0\quad\forall \y\in\Delta_d\right\}\\
&=&\left\{f\in\R[\x]_d\::\:\int_\K f\,d\phi\geq0\quad\forall \phi\in M(\K)\right\}\\
&\Rightarrow&\Delta_d^*\subseteq C_{d}(\K),
\end{eqnarray*}
and so $\Delta_d^*=C_d(\K)$. Hence the result follows if one proves that $\Delta_d$ is closed, because then 
$C_d(\K)^*=(\Delta_d^*)^*=\Delta_d$, the desired result. So let $(\y^k)\subset\Delta_d$, $k\in\N$, with $\y^k\to\y$ as $k\to\infty$.
Equivalently, $\int_\K\x^\alpha d\phi_k\to y_\alpha$ for all $\alpha\in\N^n_d$.
In particular, the convergence $y^k_0\to y_0$ implies that the sequence of measures $(\phi_k)$, $k\in\N$, is bounded,
that is, $\sup_k\phi_k(\K)<M$ for some $M>0$. As $\K$ is compact, the unit ball of $M(\K)$ is sequentially compact in the weak $\star$ topology $\sigma(M(\K),C(\K))$ where $C(\K)$ is the space of continuous functions on $\K$. Hence there is a finite Borel measure 
$\phi\in M(\K)$ and a subsequence $(k_i)$ such that $\int_\K gd\phi_{k_i}\to\int_\K gd\phi$ as $i\to\infty$, for all $g\in C(\K)$.
In particular, for every $\alpha\in\N^n_d$,
\[y_\alpha\,=\,\lim_{k\to\infty}y^{k}_\alpha\,=\,\lim_{i\to\infty}y^{k_i}_\alpha
\,=\,\lim_{i\to\infty}\int_\K\x^\alpha d\phi_{k_i}\,=\,\int_\K \x^\alpha d\phi,\]
which shows that $\y\in\Delta_d$, and so $\Delta_d$ is closed.
\end{proof}
\end{document}